\documentclass[11pt]{article}

\usepackage{amsthm,amsmath,amssymb}
\usepackage[round]{natbib}

\usepackage{hyperref}
\hypersetup{
    colorlinks,%
    citecolor=black,%
    filecolor=black,%
    linkcolor=black,%
    urlcolor=black
} 
\usepackage{hypernat}
\usepackage[top=1in, bottom=1in, left=1in, right=1in]{geometry}
\usepackage{bigints}
\usepackage{verbatim}
\usepackage{graphicx}
\usepackage{caption}
\usepackage{subcaption}
\usepackage{bbm}
\usepackage{color}

\newtheorem{theorem}{Theorem}[section]

\newtheorem{lemma}[theorem]{Lemma}
\newtheorem{corollary}[theorem]{Corollary}
\newtheorem{remark}{Remark}[section]
\newtheorem{example}[theorem]{Example}

\newcommand{\E}{{\mathbb E}}

\newcommand{\R}{{\mathbb R}}
\renewcommand{\P}{{\mathbb P}}

\newcommand{\Hs}{{\mathcal{H}}}

\newcommand{\I}{{\mathcal{I}}}

\newcommand{\K}{{\mathcal{K}}}

\newcommand{\Ps}{{\mathcal{P}}}

\newcommand{\RNum}[1]{\uppercase\expandafter{\romannumeral #1\relax}}
\makeatother
\newcommand{\argmin}{\mathop{\mathrm{arg\,min}{}}}

\newcounter{rcnt}[section]

\def\qt#1{\qquad\text{#1}}

\def\argmin{\mathop{\rm argmin}}

\definecolor{DSgray}{cmyk}{0,1,0,0}

\begin{document}

\title{Adaptive Estimation of Planar Convex Sets} 
\author{Tony Cai$^1$,  Adityanand Guntuboyina$^2$, and  Yuting Wei$^2$}

\footnotetext[1]{Department of Statistics, The Wharton School, University of Pennsylvania. The research of Tony Cai 
was supported in part by NSF Grants DMS-1208982 and DMS-1403708, and NIH Grant R01 CA127334.}
\footnotetext[2]{Department of Statistics, University of California at Berkeley. The research of Adityanand Guntuboyina was supported by NSF Grant DMS-1309356.}

\maketitle

\begin{abstract} 
In this paper, we consider adaptive estimation of  an unknown planar compact, convex set from noisy measurements of its support function on a uniform grid. Both the problem of estimating the support function at a point and that of estimating the convex set are studied. Data-driven adaptive estimators are proposed and their optimality properties are established. For pointwise estimation, it is shown that the estimator optimally adapts to every compact, convex set instead of a collection of large parameter spaces as in the conventional minimax theory in nonparametric estimation literature. For set estimation, the estimators adaptively achieve the optimal rate of convergence. In both these problems, our analysis makes no smoothness assumptions on the unknown sets. 
\end{abstract}




\section{Introduction}

We study in this paper the problem of nonparametric estimation of an unknown planar compact, convex set from noisy
measurements of its support function. Before describing the details of the problem, let us first introduce the support function. For a compact, convex set $K$ in $\R^2$, its support function is defined by  
\begin{equation*}
  h_K(\theta) := \max_{(x_1, x_2) \in K} \left(x_1 \cos \theta + x_2  \sin \theta \right) \qt{for $\theta \in \R$}. 
\end{equation*}
Note that $h_K$ is a periodic function with period $2\pi$. It is
useful to think about $\theta$ in terms of the direction $(\cos
\theta, \sin \theta)$. The line $x_1 \cos \theta + x_2 \sin \theta =
h_K(\theta)$ is a support line for $K$ (i.e., it touches $K$ and $K$
lies on one side of it). Conversely, every support line of $K$ is of
this form for some $\theta$. The convex set $K$ is completely determined by
the its support function $h_K$ because $K = \bigcap_{\theta} \{(x_1, x_2): x_1 \cos
\theta + x_2 \sin \theta \leq h_K(\theta)\}$. 

The support function $h_K$ possesses  the {\it circle-convexity} property  (see, e.g.,~\cite{VitaleCC}): 
for every $\alpha_1 > \alpha > \alpha_2$ and $0 < \alpha_1 - \alpha_2 < \pi$, 
\begin{equation}\label{genvit}
  \frac{h_K(\alpha_1)}{\sin(\alpha_1 - \alpha)} + \frac{h_K(\alpha_2)}{\sin (\alpha - \alpha_2)} \geq
  \frac{\sin(\alpha_1 - \alpha_2)}{\sin(\alpha_1 - \alpha) \sin(\alpha - \alpha_2)} h_K(\alpha). 
\end{equation}
Moreover the above inequality characterizes $h_K$, i.e., any periodic function of period $2\pi$ satisfying the above inequality equals  $h_K$ for a unique compact, convex subset $K$ in $\R^2$. The circle-convexity property \eqref{genvit} is clearly related to the usual convexity property. Indeed, if we replace the sine function in~\eqref{genvit}
by the identity function (i.e., if we replace $\sin \alpha$ by $\alpha$ in~\eqref{genvit}), we obtain the condition for convexity. In spite of this similarity,~\eqref{genvit} is different from convexity as can be seen from the example of the function $h(\theta) = |\sin \theta|$ which satisfies~\eqref{genvit} but is clearly not convex.

\subsection{The Problem, Motivations, and Background}

We are now ready to describe the problem studied in this paper. Let
$K^*$ be an unknown compact, convex set in $\R^2$. We study the 
problem of estimating $K^*$ or $h_{K^*}$ from noisy measurements of
$h_{K^*}$. Specifically, we observe data $(\theta_1, Y_1), \dots,
(\theta_n, Y_n)$ drawn according to the model
\begin{equation}
  \label{regmod}
Y_i = h_{K^*}(\theta_i) + \xi_i \qt{for $i = 1, \dots, n$}
\end{equation}
where $\theta_1, \dots, \theta_n$ are fixed grid points in $(-\pi,
\pi]$ and $\xi_1, \dots, \xi_n$ are i.i.d Gaussian random variables
with mean zero and known variance $\sigma^2$. We focus on the dual
problems of estimating the scalar quantity $h_{K^*}(\theta_i)$ for
each $1 \leq i \leq n$ as well as the convex set $K^*$. We propose data-driven adaptive estimators and establish their optimality for both of these problems. 

The problem considered here has a range of applications in engineering. The regression model \eqref{regmod} was first proposed and studied by \citet{PrinceWillskyIEEE} who were motivated by an application to Computed Tomography. \citet{LeleKulkarniWillsky} showed how solutions to this problem can be applied to target reconstruction from resolved laser-radar measurements in the presence
of registration errors. \citet{GregorRannou} considered application to Projection Magnetic Resonance Imaging. It is also a fundamental problem in the field of geometric tomography;  see~\citet{GardnerBook}. Another application domain where this problem might plausibly arise is robotic tactical sensing as has been suggested by \citet{PrinceWillskyIEEE}.  Finally this is a very natural shape constrained estimation problem and would fit right into the recent literature on shape constrained estimation. See, for example, \citet{groeneboom2014nonparametric}.   

Most proposed procedures for estimating $K^*$ in this setting are
based on least squares minimization. The least squares estimator
$\hat{K}_{\text{ls}}$ is defined as any minimizer of $\sum_{i=1}^n (Y_i -
h_{K}(\theta_i))^2$ as $K$ ranges over all compact convex sets. The
minimizer in this optimization problem is not unique and one can
always take it to be a polytope. This estimator was first proposed
by~\cite{PrinceWillskyIEEE} who also proposed an algorithm for
computing it based on quadratic programming. Further algorithms for
computing $\hat{K}_{\text{ls}}$ were proposed in~\citet{PrinceWillskyIEEE,
  LeleKulkarniWillsky,  GardnerKiderlen2009}. 

The theoretical performance of the least squares estimator was first considered by \citet{GKM06} who mainly studied its accuracy for estimating $K^*$ under the natural fixed design loss:   
\begin{equation}\label{glos}
  L_f(K^*, \hat{K}_{\text{ls}})  := \frac{1}{n} \sum_{i=1}^n
  \left(h_{K^*}(\theta_i) - h_{\hat{K}_{\text{ls}}}(\theta_i) \right)^2.   
\end{equation}
The key result of \citet{GKM06} (specialized to the planar case that we are studying) states that $L_f(K^*, \hat{K}_{\text{ls}}) = O(n^{-4/5})$ as $n \rightarrow \infty$ almost surely provided $K^*$ is contained in a ball of bounded radius. This result is complemented by the minimax lower bound in \citet{G11} where it was shown that $n^{-4/5}$ is the minimax rate for this problem. These two results together imply minimax optimality of $\hat{K}_{\text{ls}}$ under the loss function $L_f$. No other theoretical results for this problem are available outside of those in \citet{GKM06} and \citet{G11}. 

As a result, the following basic questions are still unanswered:  
\begin{enumerate}
\item For a fixed $i \in \{1, \dots, n\}$, how does one optimally and adaptively estimate $h_{K^*}(\theta_i)$? This is the pointwise estimation problem. In the literature on shape constrained estimation, pointwise estimation has been the most studied problem. Several  papers have been written on this for monotonicity constrained estimation; prominent examples being \cite{Brunk70, Wright81, G83, G85, CD99, Cator2011, H14} and convexity constrained estimation; prominent ones being \cite{HanPled76, Mammen91,  GroeneboomJongbloedWellner2001a, GroeneboomJongbloedWellner2001b, CaiLowFwork}. For the problem considered in this paper however, nothing is  known about pointwise estimation. It may be noted that the result $L_f(K^*, \hat{K}_{\text{ls}}) = O(n^{-4/5})$ of \citet{GKM06} does not say anything about the accuracy of $h_{\hat{K}_{\text{ls}}}(\theta_i)$ as an estimator for $h_{K^*}(\theta_i)$. 
  
\item How to construct minimax optimal estimators for the set $K^*$ that also adapt to polytopes? Polytopes with a small number of extreme points have a much simpler structure than general convex sets. In the problem of estimating convex sets under more standard  observation models different from the one studied here, it is possible to construct estimators that converge at faster rates for polytopes compared to the overall minimax rate (see \citet{brunel2014non} for a nice summary of this theory). Similar kinds of adaptation has been recently studied for shape constrained estimation problems based on monotonicity and convexity, see \cite{GSvex, GuntuAnnIso, baraud2015rates}. Based on these results, it is natural to expect minimax estimators that adapt to polytopes in this problem. This has not been addressed previously. 
\end{enumerate}

\subsection{Our Contributions}

We answer both the above questions in the affirmative in the present paper. The main contributions of this paper can be summarized in the following: 

\begin{enumerate}
\item We study the pointwise adaptive estimation problem in detail in
  the decision theoretic framework where the focus is on the
  performance at every function, instead of the maximum risk over a
  large parameter space. This framework, first introduced in
  \cite{cai2013adaptive} and \cite{CaiLowFwork} for shape constrained
  regression, provides a much more precise characterization of the
  performance of an estimator than the conventional minimax theory
  does.  

In the context of the present problem, the difficulty of estimating
$h_{K^*}(\theta_i)$ at a given $K^*$ and $\theta_i$ can be expressed
by means of a benchmark $R_n(K^*, \theta)$ which is defined as follows
(below $\E_{L}$ denotes expectation taken with respect to the joint
distribution of $Y_1, \dots, Y_n$ generated according to the model
\eqref{regmod} with  $K^*$ replaced by $L$): 
\begin{equation}\label{R_n}
R_n(K^*, \theta) = \sup_{L} \inf_{\tilde{h}} \max \left( \E_{K^*}
  (\tilde{h} - h_{K^*}(\theta))^2,\;  \E_{L} (\tilde{h} -
  h_{L}(\theta))^2 \right), 
\end{equation}
where the supremum above is taken over all compact, convex sets $L$
while the infimum is over all estimators $\tilde{h}$. In our first
result for pointwise estimation, we establish, for each $i \in \{1,
\dots, n\}$, a lower bound for the performance of every estimator for
estimating $h_{K^*}(\theta_i)$. Specifically,  it is shown that  
  \begin{equation}\label{co.2}
 R_n(K^*, \theta_i)  \geq  c \cdot \frac{ \sigma^2}{k_*(i)+1}
  \end{equation}
where $k_*(i)$ is an integer for which an explicit formula can be
given in terms of $K^*$ and $i$; and $c$ is a universal positive
constant. It will turn out that $k_*(i)$ is related to the smoothness
of $h_{K^*}(\theta)$ at $\theta = \theta_i$.   

We construct a data-driven estimator, $\hat{h}_i$, of
$h_{K^*}(\theta_i)$ based on local smoothing together with an
optimization scheme for automatically choosing a bandwidth, and show
that the estimator $\hat{h}_i$ satisfies 
\begin{equation}\label{co.1}
    \E_{K^*} \left(\hat{h}_i - h_{K^*}(\theta_i) \right)^2 \leq C \cdot \frac{ \sigma^2}{k_*(i) + 1} 
  \end{equation}
for a universal positive constant $C$. 
Inequalities \eqref{co.2} and \eqref{co.1} together imply  that
$\hat{h}_i$ is, within a universal constant factor,  an optimal
estimator of  $h_{K^*}(\theta_i)$ for every compact, convex set
$K^*$. This optimality is stronger than the traditional minimax
optimality usually employed in nonparametric function estimation. The
quantity $\sigma^2/(k_*(i) + 1)$ depends on the unknown set $K^*$ in a
similar way that the Fisher information bound depends on the unknown
parameter in a regular parametric model. In contrast, the optimal rate
in the minimax paradigm is given in terms of the worse case
performance over a large parameter space and does not depend on
individual parameter values.  
  
\item Using the optimal adaptive point estimators $\hat{h}_1, \dots,
  \hat{h}_n$, we construct two set estimators $\hat{K}$ and
  $\hat{K}'$. The details of this construction are given in Section
  \ref{setes}. In Theorems \ref{nt} and \ref{nti}, we prove that
  $\hat{K}$ is minimax optimal for $K^*$ under the loss function $L_f$
  while the estimator $\hat{K}'$ is minimax optimal under the integral
  squared loss function defined by   
\begin{equation}\label{id}
  L(\hat{K}', K^*) := \int_{-\pi}^{\pi} \left(h_{\hat{K}'}(\theta) -
    h_{K^*}(\theta) \right)^2 d\theta. 
\end{equation}
Specifically,  Theorem \ref{nt} shows that
\begin{equation}\label{co.3}
  \E_{K^*} L_f(K^*, \hat{K}) \leq C \left\{\frac{\sigma^2}{n} +
    \left(\frac{\sigma^2 \sqrt{R}}{n} \right)^{4/5}\right\}
\end{equation}
provided $K^*$ is contained in a ball of radius $R$. This, combined
with the minimax lower bound in \citet{G11}, proves the minimax
optimality of $\hat{K}$. An analogous result is shown in
Theorem~\ref{nti} for $\E_{K^*} L(K^*, \hat{K}')$. For the pointwise
estimation problem where the goal is to estimate $h_{K^*}(\theta_i)$,
the optimal rate $\sigma^2/(k_*(i) + 1)$ can be as large as
$n^{-2/3}$. However the bound \eqref{co.3} shows that the globally the
risk is at most $n^{-4/5}$. The shape constraint given by convexity of
$K^*$ ensures that the points where pointwise estimation rate is
$n^{-2/3}$ cannot be too many. Note that we make no smoothness
assumptions for proving \eqref{co.3}.  

\item We show that our set estimators $\hat{K}$ and $\hat{K}'$ adapt
  to polytopes with bounded number of extreme points. Already
  inequality \eqref{co.3} implies that $\E_{K^*} L_f(K^*, \hat{K})$ is
  bounded from above by the parametric risk $C \sigma^2/n$ provided $R
  = 0$ (note that $R = 0$ means that $K^*$ is a singleton). Because
  $\sigma^2/n$ is much smaller than $n^{-4/5}$, the bound \eqref{co.3}
  shows that $\hat{K}$ adapts to singletons.  Theorem
  \ref{theorem::General} extends this adaptation phenomenon to
  polytopes and we show that $\E_{K^*} L_f(K^*, \hat{K})$ is bounded
  by the parametric rate (up to a logarithmic multiplicative factor of
  $n$) for all polytopes with bounded number of extreme points. An
  analogous result is also proved for $\E_{K^*} L(K^*, \hat{K}')$ in
  Theorem \ref{nti}. It should be noted that the construction of our
  estimators $\hat{K}$ and $\hat{K}'$ (described in Section
  \ref{setes}) does not involve any special treatment for polytopes;
  yet the estimators automatically achieve faster rates for polytopes.  

\end{enumerate}

We would like to stress two features of this paper: (a) we do not make any smoothness assumptions on the boundary of $K^*$ throughout the paper; in particular, note that we obtain the $n^{-4/5}$ rate for the set estimators $\hat{K}$ and $\hat{K}'$ without any smoothness assumptions, and (b) we go beyond the traditional minimax paradigm by considering adaptive estimation in both the pointwise estimation problem and the problem of estimating the entire set $K^*$. 

\subsection{Organization of the Paper}

The rest of the paper is structured as follows. The proposed estimators are described in detail in Section~\ref{est}. The theoretical properties of the estimators are analyzed in Section \ref{locb}; Section \ref{cors} gives results for pointwise estimation while Section \ref{setac} deals with set estimators. In Section \ref{egs}, we investigate optimal estimation of some special compact convex sets $K^*$ where we explicitly compute the associated rates of convergence. The proofs of the main results are given in Section \ref{pmr} and additional technical results are relegated to Appendix \ref{apap}.


\section{Estimation Procedures}
\label{est}

Recall the regression model \eqref{regmod}, where we observe noisy
measurements $(\theta_1, Y_1), \dots,(\theta_n, Y_n)$ with 
\mbox{$\theta_i =
2\pi i/n-\pi, ~i=1,...,n$} being fixed grid points in $(-\pi,
\pi]$. In this section, we first describe in detail our estimate
$\hat{h}_i$ for $h_{K^*}(\theta_i)$ for each $i$. Subsequently,
we shall describe how to put together these estimates $\hat{h}_1,
\dots, \hat{h}_n$ to yield set estimators for $K^*$.  

\subsection{Estimators for $h_{K^*}(\theta_i)$ for each fixed $i$}

Fix $1 \leq i \leq n$. Our construction of the estimator $\hat{h}_i$
for $h_{K^*}(\theta_i)$ is based on the key circle-convexity property
\eqref{genvit} of the function $h_{K^*}(\cdot)$. Let us define, for $0
< \phi < \pi/2$ and $\theta \in (-\pi, \pi]$, the following two
quantities:    
\begin{equation*}
l(\theta, \phi) := \cos \phi \left(h_{K^*}(\theta+\phi) + h_{K^*}(\theta-\phi) \right)   -
\frac{h_{K^*}(\theta+2\phi) + h_{K^*}(\theta-2\phi)}{2}
\end{equation*}
and 
\begin{equation*}
u(\theta,\phi) := \frac{h_{K^*}(\theta+\phi) + h_{K^*}(\theta-\phi)}{2\cos \phi}.  
\end{equation*}
The following lemma states that for every $\theta$, the quantity
$h_{K^*}(\theta)$ is sandwiched between $l(\theta, \phi)$ and
$u(\theta, \phi)$ for every~$\phi$. This will be used crucially in
defining $\hat{h}$. The proof of this lemma is a straightforward
consequence of \eqref{genvit} and is given in Appendix \ref{apap}.    

\begin{lemma}\label{genvitthm}
For every $0 < \phi < \pi/2$ and every $\theta \in(-\pi,\pi]$, we have
$l(\theta,\phi) \leq h_{K^*}(\theta) \leq u(\theta,\phi)$.
\end{lemma}

For a fixed $1 \leq i \leq n$, Lemma~\ref{genvitthm} implies that
$l(\theta_i,\frac{2\pi j}{n}) \leq h_{K^*}(\theta_i) \leq
u(\theta_i,\frac{2\pi j}{n})$ for every \mbox{$0 \leq j < \lfloor n/4
\rfloor$}. Note that when $j=0$, we have $l(\theta_i,0) = 
h_{K^*}(\theta_i) = u(\theta_i,0).$ Averaging these inequalities for
$j = 0, 1, \dots, k$ where $k$ is a fixed integer 
with $0 \leq k < \lfloor n/4 \rfloor$, we obtain
\begin{equation}
  \label{kiest}
L_k(\theta_i) \leq h_{K^*}(\theta_i) \leq U_k(\theta_i)  \qt{for every
$0 \leq k < \lfloor n/4 \rfloor$}
\end{equation}
where
\begin{equation*}
L_k(\theta_i) := \frac{1}{k+1} \sum_{j = 0}^k l \left(\theta_i, \frac{2
      \pi j}{n} \right)  ~~ \text{ and } ~~ U_k(\theta_i) := \frac{1}{k+1} \sum_{j =
    0}^k u \left(\theta_i, \frac{2 \pi j}{n} \right). 
\end{equation*}

We are now ready to describe our estimator. Fix $1
\leq i \leq n$. Inequality~\eqref{kiest} says that the quantity of
interest, $h_{K^*}(\theta_i)$, is sandwiched between $L_k(\theta_i)$ and 
$U_k(\theta_i)$ for every $k$. Both $L_k(\theta_i)$ and
$U_k(\theta_i)$ can naturally be estimated by unbiased
estimators. Indeed, let 
\begin{equation*}
  \hat{l}(\theta_i, 2j\pi/n) := \cos (2j \pi/n) (Y_{i+j} + Y_{i-j}) - \frac{Y_{i+2j} +
    Y_{i-2j}}{2} ~~\text{   and   } ~~ \hat{u}(\theta_i, 2j\pi/n) := \frac{Y_{i+j} +
    Y_{i-j}}{2 \cos (2j \pi/n)}
\end{equation*}
and take
\begin{equation}\label{avig}
 \hat{L}_k(\theta_i) := \frac{1}{k+1}
  \sum_{j = 0}^k \hat{l} \left(\theta_i, 2j\pi/n \right) ~~ \text{ and }
  ~~ \hat{U}_k(\theta_i) := \frac{1}{k+1} \sum_{j = 0}^k \hat{u}
  \left(\theta_i, 2j\pi/n \right). 
\end{equation}
Obviously, in order for the above to be meaningful, we need to define
$Y_i$ even for $i \notin \{1, \dots, n\}$. This is easily done in the
following way: for any $i \in \mathbb{Z}$, let $s$ be such that $i -
sn \in \{1, \dots, n\}$ and take $Y_i := Y_{i-sn}$. 

As $k$ increases, one averages more terms in \eqref{avig} and hence
the estimators $\hat{L}_k(\theta_i)$ and $\hat{U}_k(\theta_i)$ become
more accurate. Let 
\begin{equation}\label{deluserep}
  \hat{\Delta}_k(\theta_i) := \hat{U}_k(\theta_i) -
  \hat{L}_k(\theta_i) = \frac{1}{k+1} \sum_{j=0}^k \left(
    \frac{Y_{i+2j} + Y_{i-2j}}{2} - \frac{\cos (4 j \pi/n)}{\cos (2j
      \pi/n)} \frac{Y_{i+j} + Y_{i-j}}{2} \right). 
\end{equation}
Because of \eqref{kiest}, a natural strategy for estimating 
$h_{K^*}(\theta_i)$ is to choose $k$ for which
$\hat{\Delta}_k(\theta_i)$ is the smallest and then use either
$\hat{L}_{k}(\theta_i)$ or $\hat{U}_{k}(\theta_i)$ at that $k$ as the
estimator. This is essentially our estimator with one small difference
in that we also take into account the noise present in
$\hat{\Delta}_k(\theta_i)$. Formally, our estimator for
$h_{K^*}(\theta_i)$ is given by:  
\begin{equation}\label{estimator}
 \hat{h}_i = \hat{U}_{\hat{k}(i)}(\theta_i),~\mbox{where}~\hat{k}(i)
 := \argmin_{k \in \I} \left\{ \left(\hat{\Delta}_k(\theta_i)
   \right)_+  + \frac{2\sigma }{\sqrt{k+1}} \right\} 
\end{equation}
and $\I := \{0\}\cup\{ 2^j : j \geq 0 \text{ and } 2^j \leq \lfloor
n/16\rfloor \}$.  

Our estimator $\hat{h}_i$ can be viewed as an angle-adjusted local
averaging estimator. It is inspired by the estimator
of~\citet{CaiLowFwork} for convex regression. The number of terms
averaged equals $\hat{k}(i) + 1$ and this is analogous to the bandwidth
in kernel-based smoothing methods. Our $\hat{k}(i)$ is determined from
an optimization scheme. Notice that unlike the least squares estimator
$h_{\hat{K}_{\text{ls}}}(\theta_i)$, the construction of $\hat{h}_i$
for a fixed $i$ does not depend on the construction of $\hat{h}_j$ for
$j\neq i$.

\subsection{Set Estimators for $K^*$} \label{setes}

We next present estimators for the set $K^*$. The point estimators 
$\hat{h}_1, \dots, \hat{h}_n$ do not directly give an estimator for
$K^*$ because $(\hat{h}_1, \dots, \hat{h}_n)$ is not necessarily a
valid support vector i.e., $(\hat{h}_1, \dots, \hat{h}_n)$ does not
always belong to the following set:
\begin{equation*}
  \Hs := \left\{(h_{K}(\theta_1), \dots, h_{K}(\theta_n)) : K
    \subseteq \R^2 \text{ is compact and convex} \right\}. 
\end{equation*}
To get a valid support vector from $(\hat{h}_1, \dots, \hat{h}_n)$, we
need to  project it onto $\Hs$ to obtain: 
\begin{equation}\label{vsp}
    \hat{h}^{P} := (\hat{h}^P_1, \dots, \hat{h}^P_n) := \argmin_{(h_1,
      \dots, h_n) \in \Hs} \sum_{i=1}^n \left(\hat{h}_i - h_i
    \right)^2  
  \end{equation}
The superscript $P$ here stands for projection. An estimator for the
set $K^*$ can now be constructed immediately from $\hat{h}^P_1, \dots,
\hat{h}_n^P$ via  
\begin{equation}\label{khdef}
  \hat{K} := \left\{(x_1, x_2) : x_1 \cos \theta_i +
    x_2 \sin \theta_i \leq \hat{h}_i^P \text{ for all } i = 1, \dots,
    n \right\}.  
\end{equation}
In Theorems \ref{nt} and \ref{theorem::General}, we prove upper
bounds on the accuracy of $\hat{K}$ under the loss function $L_f$
defined in \eqref{glos}. 

There is another reasonable way of constructing a set estimator for
$K^*$ based on the point estimators $\hat{h}_1, \dots,
\hat{h}_n$. We first interpolate $\hat{h}_1, \dots, \hat{h}_n$ to
define a function $\hat{h}' : (-\pi, \pi] \rightarrow \R$ as follows: 
\begin{equation}\label{hind}
  \hat{h}'(\theta) := \frac{\sin (\theta_{i+1} - \theta)}{\sin
    (\theta_{i+1} - \theta_i)} \hat{h}_i + \frac{\sin(\theta -
    \theta_i)}{\sin(\theta_{i+1} - \theta_i)} \hat{h}_{i+1} \qt{for
    $\theta_i \leq \theta \leq \theta_{i+1}$}. 
\end{equation}
Here $i$ ranges over $1, \dots, n$ with the convention that
$\theta_{n+1} = \theta_1 + 2\pi$ (and $\theta_n \leq \theta \leq
\theta_{n+1}$ should be identified with $-\pi \leq \theta \leq -\pi + 2
\pi/n$). Based on this function $\hat{h}'$, we can define our estimator
$\hat{K}'$ of $K^*$ by 
\begin{equation}\label{khpdef}
  \hat{K}' := \argmin_{K} \int_{-\pi}^{\pi} \left(\hat{h}'(\theta) -
    h_{K}(\theta) \right)^2 d\theta . 
\end{equation}
The existence and uniqueness of $\hat{K}'$ can be justified in the
usual way by the Hilbert space projection theorem. In Theorem 
\ref{nti}, we prove bounds on the accuracy of $\hat{K}'$ as an
estimator for $K^*$ under the integral loss $L$ defined in
\eqref{id}.


\section{Main Results}\label{locb}  

We investigate in this section the accuracy of the proposed point and set estimators. The proofs of these results are given in Section \ref{pmr}.

\subsection{Accuracy of the Point Estimator}\label{cors}

As mentioned in the introduction, we evaluate the performance of the
point estimator $\hat{h}_i$ at individual functions, not the worst
case over a large parameter space. This provides a much more precise
characterization of the accuracy of the estimator.   
 Let us first recall inequality \eqref{kiest} where $h_{K^*}(\theta_i)$ is sandwiched between $L_k(\theta_i)$ and  $U_k(\theta_i)$. Define $\Delta_k(\theta_i) := U_k(\theta_i) - L_k(\theta_i)$. 

\begin{theorem}\label{rbe}
Fix $i \in \{1, \dots, n\}$. There exists a universal positive constant $C$ such that the risk of $\hat{h}_i$ as an estimator of $h_{K^*}(\theta_i)$ satisfies the following inequality:  
  \begin{equation}\label{rbe.eq}
    \E_{K^*} \left(\hat{h}_i - h_{K^*}(\theta_i) \right)^2 \leq C \cdot \frac{ \sigma^2}{k_*(i)+1} 
  \end{equation}
where 
\begin{equation}\label{kst}
  k_*(i) := \argmin_{k \in \I} \left(\Delta_k(\theta_i) + \frac{2\sigma}{\sqrt{k+1}} \right).
\end{equation} 
\end{theorem}

\begin{remark}{\rm
 It turns out that the bound in \eqref{rbe.eq} is linked to the level of smoothness of the function $h_{K^*}$ at $\theta_i$. However for this interpretation to be correct, one needs to regard $h_{K^*}$ as a function on $\R^2$ instead of a subset of $\R$. This is further explained in Remark \ref{remco}. 
}
\end{remark}

Theorem \ref{rbe} gives an explicit bound on the risk of $\hat{h}_i$ in terms of the quantity $k_*(i)$ defined in \eqref{kst}. It is important to keep in mind that $k_*(i)$ depends on $K^*$ even though this is suppressed in the notation. In the next theorem, we show that $\sigma^2/(k_*(i) + 1)$ also presents a lower bound on the accuracy of every estimator for $h_{K^*}(\theta_i)$. This implies, in particular, optimality of $\hat{h}_i$ as an estimator of $h_{K^*}(\theta_i)$. 

One needs to be careful in formulating the lower bound result in this setting. A first attempt might perhaps be to prove that, for a universal positive constant $c$, 
\begin{equation*}
  \inf_{\tilde{h}} \E_{K^*}\left(\tilde{h} - h_{K^*}(\theta_i) \right)^2 \geq c\cdot \frac{ \sigma^2}{k_*(i) + 1} 
\end{equation*}
where the infimum is over all possible estimators $\tilde{h}$. This, of course, would not be possible because one can take $\tilde{h} = h_{K^*}(\theta_i)$ which would make the left hand side above zero. A formulation of the lower bound which avoids this difficulty was
proposed by \cite{CaiLowFwork} in the context of convex function estimation. Their idea, translated to our setting  of estimating the support function $h_{K^*}$ at a point $\theta_i$, is to consider, instead of the risk at $K^*$, the maximum of the risk at $K^*$ and the risk at $L^*$  which  is most difficult to distinguish from $K^*$ in term of estimating $h_{K^*}(\theta_i)$. This leads to the benchmark $ R_n(K^*, \theta_i)$ defined in \eqref{R_n}.
\begin{theorem}\label{lobo}
For any fixed $i \in \{1, \dots, n\}$,
we have 
\begin{equation}
\label{lobo.eq}
R_n(K^*, \theta_i) \geq c\cdot \frac{ \sigma^2}{ k_*(i)+1}
\end{equation}
for a universal positive constant $c$. 
\end{theorem}

Theorems \ref{rbe}  and \ref{lobo} together imply that $\sigma^2/(k_*(i) + 1)$ is the optimal rate of estimation of $h_{K^*}(\theta_i)$ for a given compact, convex set $K^*$. 
The results show that our data driven estimator $\hat{h}_i$ for $h_{K^*}(\theta_i)$ performs uniformly within a constant factor of the ideal benchmark $R_n(K^*, \theta_i)$ for every $i$. This means that $\hat{h}_i$ adapts to every unknown set $K^*$ instead of a collection of large parameter spaces as in the conventional minimax theory commonly used in nonparametric literature.

Given a specific set $K^*$ and $1 \leq i \leq n$, the quantity $k_*(i)$ is often straightforward to compute up to constant multiplicative factors. Several examples are provided in Section \ref{egs}. From these examples, it will be clear that the size of $\sigma^2/(k_*(i) +
1)$ is linked to the level of smoothness of the function $h_{K^*}$ at $\theta_i$. However for this interpretation to be correct, one needs
to regard $h_{K^*}$ as a function on $\R^2$ instead of a subset of $\R$. This is explained in Remark \ref{remco}.

The following corollaries shed more light on the quantity
$\sigma^2/(k_*(i) + 1)$. The first corollary below shows that
$\sigma^2/(k_*(i) + 1)$ is at most $C (\sigma^2 R/n)^{-2/3}$ for every
$i$ and $K^*$ ($C$ is a universal constant). This implies, in
particular, the consistency of 
$\hat{h}_i$ as an estimator for $h_{K^*}(\theta_i)$ for every $i$ and
$K^*$. In Example \ref{segm}, we provide an explicit choice of $i$ and
$K^*$ for which $\sigma^2/(k_*(i) + 1) \geq c (\sigma^2 R/n)^{-2/3}$
($c$ is a universal constant). This implies that the conclusion of the
following corollary cannot in general be improved. 

\begin{corollary}\label{theorem::RadiusRBall}
Suppose $K^*$ is contained in some closed ball of radius $R$. Then for 
every $i \in \{1, \dots, n\}$, we have 
\begin{equation}
  \label{f1.eq}
  \frac{\sigma^2}{k_*(i) + 1} \leq C \left(\frac{\sigma^2 R}{n}
  \right)^{2/3} 
\end{equation}
and 
\begin{equation}\label{lp1.eq}
  \E \left(\hat{h}_i - h_{K^*}(\theta_i)\right)^2 \leq C \left(\frac{\sigma^2
      R}{n} \right)^{2/3} . 
\end{equation}
for a universal positive constant $C$. 
\end{corollary}

It is clear from the definition \eqref{kst}  that $k_*(i) \leq n$ for
all $i$ and $K_*$. In the next corollary, we prove that there exist
sets $K_*$ and $i$ for which $k_*(i) \geq c n$ for a constant
$c$. For these sets, the optimal rate of estimating
$h_{K^*}(\theta_i)$ is therefore parametric. 

For a fixed $i$ and $K^*$, let $\phi_1(i)$ and $\phi_2(i)$ be
such that $\phi_1(i) \leq \theta_i \leq \phi_2(i)$ and such that there
exists a single point $(x_1, x_2) \in K^*$ with  
\begin{equation}\label{anny}
  h_{K^*}(\theta) = x_1 \cos \theta + x_2 \sin \theta \qt{for all
    $\theta \in [\phi_1(i), \phi_2(i)]$}. 
\end{equation}
The following corollary says that if the distance of $\theta_i$ to its
nearest end-point in the interval $[\phi_1(i), \phi_2(i)]$ is large
(i.e., of constant order), then the optimal rate of estimation of
$h_{K^*}(\theta_i)$ is parametric. This situation happens usually for polytopes
(polytopes are compact, convex sets with finitely many vertices); see
Examples \ref{sp} and \ref{segm} for specific instances of this
phenomenon. For non-polytopes, it can often happen that $\phi_1(i) =
\phi_2(i) = \theta_i$ in which case the conclusion of the next
corollary is not useful. 

\begin{corollary}\label{theorem::individual}
For every $i \in \{1, \dots, n\}$, we have 
\begin{equation}\label{vidu.eq}
  k_*(i) \geq c~ n \min \left(\theta_i - \phi_1(i), \phi_2(i) - \theta_i,
  \pi\right) 
\end{equation}
for a universal positive constant $c$. Consequently 
\begin{equation}\label{indi.eq}
\E \left( \hat{h}_i - h_{K^*}(\theta_i) \right)^2 \leq
\frac{C\sigma^2}{1 + n \min(\theta_i - \phi_1(i), \phi_2(i) -
  \theta_i, 
  \pi)}   
\end{equation}
for a universal positive constant C.
\end{corollary}

From the above two corollaries, it is clear that the optimal rate of estimation of $h_{K^*}(\theta_i)$ can be as large as $n^{-2/3}$ and as small as the parametric rate $n^{-1}$. The rate $n^{-2/3}$ is achieved, for example, in the situation demonstrated in Example \ref{segm} while the parametric rate is achieved, for example,  for polytopes. 

The next corollary argues that in order to bound $k_*(i)$ in specific
examples, one only needs to bound the quantity $\Delta_k(\theta_i)$
from above and below. This corollary will be very useful in Section
\ref{egs} while working out $k_*(i)$ in specific examples.   

\begin{corollary}\label{arbe}
Fix $1 \leq i \leq n$. Let $\{f_k(\theta_i), k \in \I\}$  and
$\{g_k(\theta_i), k \in \I\}$ be two sequences which satisfy  
$g_k(\theta_i) \leq \Delta_k(\theta_i) \leq f_k(\theta_i)$ for all $k
\in \I$. Also let  
\begin{equation}
  \label{arbe.def}
  \breve{k}(i) := \max \left\{k \in \I : f_k(\theta_i) < \frac{
      (\sqrt{6} - 2)\sigma}{\sqrt{k + 1}} \right\}
\end{equation}
and 
\begin{equation}
  \label{arbe_lower.def}
  \tilde{k}(i) := \min \left\{k \in \I : g_k(\theta_i) >
    \frac{6(\sqrt{2} - 1) \sigma}{\sqrt{k + 1}} \right\}
\end{equation}
as long as there is some $k \in \I$ for which $g_k(\theta_i) > 6
(\sqrt{2} - 1) \sigma/\sqrt{k+1}$; otherwise take $\tilde{k}(i) :=  
\max_{k \in \I} k$. We then have $\breve{k}(i) \leq k_*(i) \leq
\tilde{k}(i)$ and 
\begin{equation}
  \label{arbe.eq}
  \E_{K^*} \left(\hat{h}_i - h_{K^{*}}(\theta_i) \right)^2 \leq C
  \frac{\sigma^2}{\breve{k}(i) + 1}  
\end{equation}
for a universal positive constant $C$. 
\end{corollary} 

\subsection{Accuracy of Set Estimators}
\label{setac}

We now turn to study the accuracy of the set estimators $\hat{K}$ (defined in \eqref{khdef}) and $\hat{K}'$ (defined in \eqref{khpdef}). The accuracy of $\hat{K}$ will be investigated under the loss function $L_f$ (defined in \eqref{glos}) while the accuracy of $\hat{K}'$ will be studied under the loss function $L$ (defined in \eqref{id}). 

In Theorem \ref{nt} below, we prove that $\E_{K^*} L_f(K^*, \hat{K})$ is bounded from above by a constant multiple of  $n^{-4/5}$ as long as $K^*$ is contained in a ball of radius $R$. The discussions following the theorem shed more light on its implications. 

\begin{theorem}\label{nt}
If $K^*$ is contained in some closed ball of radius $R \geq 0$, we have  
  \begin{equation}\label{nt.eq}
    \E_{K^*} L_f\left(K^*, \hat{K} \right) \leq C
    \left\{\frac{\sigma^2}{n} + \left(\frac{\sigma^2 \sqrt{R}}{n}
      \right)^{4/5} \right\}
  \end{equation}
  for a universal positive constant $C$. Note here that $R = 0$ is
  allowed (in which case $K^*$ is a singleton). 
\end{theorem}

Note that as long as $R > 0$, the right hand side in \eqref{nt.eq} will be dominated by the $(\sigma^2 \sqrt{R}/n)^{-4/5}$ term for all large $n$. This would mean that 
\begin{equation}\label{acon}
    \sup_{K^* \in \K(R)}  \E_{K^*} L_f(K^*, \hat{K}) \leq C
    \left(\frac{\sigma^2 \sqrt{R}}{n} \right)^{4/5}
 \end{equation}
where $\K(R)$ denotes the set of all compact convex sets contained in some fixed closed ball of radius $R$.  

The minimax rate of estimation over the class $\K(R)$ was studied in \citet{G11}. In \citet[Theorems 3.1 and 3.2]{G11}, it was proved that 
\begin{equation}\label{fme}
\inf_{\tilde{K}} \sup_{K^* \in \K(R)} \E_{K^*} L_f(K^*, \hat{K}) \asymp \left( \frac{\sigma^2 \sqrt{R}}{n}\right)^{4/5} \end{equation}
where $\asymp$ denotes equality upto constant multiplicative factors. From \eqref{acon} and \eqref{fme}, it follows that $\hat{K}$ is a minimax optimal estimator of $K^*$. We should mention here that an inequality of the form \eqref{acon} was proved for the least squares estimator $\hat{K}_{\text{ls}}$ by \citet{GKM06} which implies that $\hat{K}_{\text{ls}}$ is also a minimax optimal estimator of $K^*$. 

The $n^{-4/5}$ minimax rate here is quite natural in connection with estimation of smooth functions. Indeed, this is the minimax rate of estimation of twice smooth one-dimensional functions. Although we have not made any smoothness assumptions here, we are working under a convexity-based constraint and convexity is associated, in a broad sense, with twice smoothness (see, for example, \citet{Alexandrov39}).

\begin{remark}{\rm
Because of the formula \eqref{glos} for the loss function $L_f$, the risk $\E_{K^*} L_f(K^*, \hat{K})$ can be seen as the average of the risk of $\hat{K}$ for estimating $h_{K^*}(\theta_i)$ over $i = 1, \dots, n$. We have seen in Section \ref{cors} that the optimal rate of estimating $h_{K^*}(\theta_i)$ can be as high as $n^{-2/3}$. Theorem \ref{nt}, on the other hand, can be interpreted as saying that, on average over $i = 1, \dots, n$, the optimal rate of estimating $h_{K^*}(\theta_i)$ is at most $n^{-4/5}$. Indeed, the key to proving Theorem \ref{nt} is to establish the following inequality:  \begin{equation*}
\frac{\sigma^2}{n} \sum_{i=1}^n \frac{1}{k_*(i) + 1} \leq C \left\{\frac{\sigma^2}{n} + \left(\frac{\sigma^2 \sqrt{R}}{n} \right)^{4/5} \right\}. 
\end{equation*}
under the assumption that $K^*$ is contained in a ball of radius $R$. Therefore, even though each term $\sigma^2/(k_*(i) + 1)$ can be as large as $n^{-2/3}$, on average, their size is at most $n^{-4/5}$. 
}
\end{remark}

\begin{remark}{\rm
Theorem \ref{nt} provides different qualitative conclusions when $K^*$ is a singleton. In this case, one can take $R = 0$ in \eqref{nt.eq} to get the parametric bound $C \sigma^2/n$ for $\E_{K^*} L_f(K^*, \hat{K})$. Because this is smaller than the nonparametric $n^{-4/5}$ rate, it means that $\hat{K}$ adapts to singletons. Singletons are simple examples of polytopes and one naturally wonders here if $\hat{K}$ also adapts to other polytopes as well. This is however not implied by inequality \eqref{nt.eq} which gives the rate $n^{-4/5}$ for every $K^*$ that is not a singleton. It turns out that $\hat{K}$ indeed adapts to other polytopes and we prove this in the next theorem. In fact, we prove that $\hat{K}$ adapts to any $K^*$ that is well-approximated by a polytope with not too many vertices. It is currently not known if the least squares estimator $\hat{K}_{\text{ls}}$ has such adaptive estimation properties. 
}
\end{remark}

In the next theorem, we prove another bound for $\E_{K^*} L_f(K^*,\hat{K})$. This bound demonstrates adaptive estimation properties of $\hat{K}$ as described in the previous remark. Before stating the theorem, we need some notation. Recall that polytopes are compact, convex sets with finitely many extreme points (or vertices). The space of all polytopes in $\R^n$ will be denoted by $\Ps$. For a polytope $P \in \Ps$, we denote by $v_P$, the number of extreme points of $P$. Also recall the notion of Hausdorff distance between two compact, convex sets $K$ and $L$ defined by 
\begin{equation}\label{hausdef}
  \ell_H(K, L) := \sup_{\theta \in \R} \left|h_{K}(\theta) - h_{L}(\theta) \right|.  
\end{equation}
This is not the usual way of defining the Hausdorff distance. For an explanation of the connection between this and the usual definition, see, for example, \citet[Theorem 1.8.11]{Schneider}. 

\begin{theorem}\label{theorem::General}
There exists a universal positive constant $C$ such that 
\begin{equation}\label{GloGen}
\E_{K^*} L_f(K^*, \hat{K}) \leq C  \inf_{P \in \Ps} 
   \left[ \frac{\sigma^2 v_P}{n} \log \left(\frac{en}{v_P} \right) +
     \ell^2_H(K^*, P)\right].  
\end{equation}
\end{theorem}

\begin{remark}[Near-parametric rates for polytopes]{\rm
  The bound \eqref{GloGen} implies that $\hat{h}$ has the parametric
  rate (upto a logarithmic factor of $n$)  for estimating
  polytopes. Indeed, suppose that $K^*$ is a polytope with $v$
  vertices. Then using $P = K^*$ in the infimum in \eqref{GloGen}, we
  have the risk bound 
  \begin{equation}\label{janak}
    \E_{K^*} L_f(K^*, \hat{K}) \leq \frac{C \sigma^2 v}{n} \log
    \left(\frac{en}{v} \right). 
  \end{equation}
  This is the parametric rate $\sigma^2 v/n$ up to logarithmic
  factors and is smaller than the nonparametric rate $n^{-4/5}$ given
  in \eqref{nt.eq}. 
}
\end{remark}

\begin{remark}{\rm
When $v = 1$, inequality \eqref{janak} has a redundant logarithmic
factor. Indeed, when $v = 1$, we can use \eqref{nt.eq}   with $R = 0$
which gives \eqref{janak} without the additional logarithmic
factor. We do not know if the logarithmic factor in \eqref{janak} can
be removed for values of $v$ larger than one as well. 
}
\end{remark}

We now turn to our second set estimator $\hat{K}'$. For this
estimator, the next theorem provides an upper bound on its accuracy
under the integral loss function $L$ (defined in
\eqref{id}). Qualitatively, the bounds on $\E_{K^*} L(K^*, \hat{K}')$
given in the next theorem are similar to the bounds on $\E_{K^*}
L_f(K^*, \hat{K})$ proved in Theorems~\ref{nt} and
\ref{theorem::General}. 
\begin{theorem}\label{nti}
  Suppose $K^*$ is contained in some closed ball of radius $R \geq
  0$. The risk $\E_{K^*} L(K^*, \hat{K}')$ satisfies both the
  following inequalities: 
  \begin{equation}\label{nti.1}
    \E_{K^*} L(K^*, \hat{K}') \leq C \left\{\frac{\sigma^2}{n} +
      \left(\frac{\sigma^2 \sqrt{R}}{n} \right)^{4/5} +
      \frac{R^2}{n^2} \right\} 
  \end{equation}
  and 
  \begin{equation}\label{nti.2}
    \E_{K^*} L(K^*, \hat{K}') \leq C \inf_{P \in \Ps}
    \left[\frac{\sigma^2 v_P}{n} \log \left(\frac{en}{v_P} \right) +
      \ell_H^2(K^*, P)  + \frac{R^2}{n^2}\right]. 
  \end{equation}
\end{theorem}

The only difference between the inequalities \eqref{nti.1} and \eqref{nti.2} on one hand and \eqref{nt.eq} and \eqref{GloGen} on the other is the presence of the $R^2/n^2$ term. This term is usually very small and does not change the qualitative behavior of the bounds. However note that inequality \eqref{GloGen} did not require any assumption on $K^*$ being in a ball of radius $R$ while this assumption is necessary for \eqref{nti.2}.

\begin{remark}{\rm
The rate $(\sigma^2 \sqrt{R}/n)^{4/5}$ is the minimax rate for this problem under the loss function $L$. Although this has not been proved explicitly anywhere, it can be shown by modifying the proof of \citet[Theorem 3.2]{G11} appropriately. Theorem \ref{nti} therefore shows that $\hat{K}'$ is a minimax optimal estimator of $K^*$ under the loss function $L$. 
}
\end{remark}


\section{Examples}
\label{egs}

We now investigate the conclusions of the theorems of the previous
section for specific choices of $K^*$. For calculations in the
following examples, it will be useful here to note that the quantity 
\mbox{$\Delta_k(\theta_i) = U_k(\theta_i) - L_k(\theta_i)$} has the following 
alternative expression: 
\begin{equation}\label{alex}
  \frac{1}{k+1} \sum_{j=0}^k
  \left(\frac{h_{K^*}(\theta_i+ 4j\pi/n) + h_{K^*}(\theta_i -4j
      \pi/n)}{2} - \frac{\cos(4j\pi/n)}{\cos(2j\pi/n)}
    \frac{h_{K^*}(\theta_i+2j\pi/n) + h_{K^*}(\theta_i -2j\pi/n)}{2}
  \right).  
\end{equation}

\begin{example}[Single point]\label{sp}{\rm
  Suppose $K^* := \left\{(x_1, x_2) \right\}$ for a fixed point $(x_1, 
  x_2) \in \R^2$. In this case 
  \begin{equation}\label{sinsp}
    h_{K^*}(\theta) = x_1 \cos \theta + x_2 \sin \theta \qt{for all
      $\theta$}. 
  \end{equation} 
  It can then be directly checked from \eqref{alex} that
  $\Delta_k(\theta_i) = 0$ for every $k \in \I$ and $i \in \{1, \dots,
  n\}$. As a result, it   follows that $k_*(i) = \max_{k \in \I} k
  \geq cn$ for a positive constant $c$. 

  Theorem \ref{rbe} then says that the point estimator $\hat{h}_i$
  satisfies 
  \begin{equation}\label{spee}
    \E_{K^*} \left( \hat{h}_i - h_{K^*}(\theta_i) \right)^2 \leq  
    \frac{C\sigma^2}{n} 
  \end{equation}
  for a universal positive constant $C$. One therefore gets the
  parametric rate here. 

  Also, Theorem \ref{nt} and inequality \eqref{nti.1} in Theorem
  \ref{nti} can both be used here with $R = 0$. This implies that the
  set estimators $\hat{K}$ and $\hat{K}'$ both converge to $K^*$ at
  the parametric rate under the loss functions $L_f$ and $L$
  respectively. 
}
\end{example}

\begin{example}[Ball]\label{balle}{\rm
  Suppose $K^*$ is a ball centered at $(x_1, x_2)$ with radius $R >
  0$. It is then easy to verify that
  \begin{equation}\label{ballsp}
    h_{K^*}(\theta) = x_1 \cos \theta + x_2 \sin \theta + R \qt{for
      all $\theta$}. 
  \end{equation}
As a result, for every $k \in \I$ and $i \in \{1, \dots,
  n\}$, we have
  \begin{equation}\label{bca}
    \Delta_k(\theta_i) = \frac{R}{k+1} \sum_{j=0}^{k} \left(1 -
      \frac{\cos \frac{4\pi j}{n}}{\cos \frac{2\pi j}{n}} \right) \leq
    R \left(1 - \frac{\cos 4\pi k/n}{\cos  2\pi k/n} \right) = \frac{R(1
      + 2\cos 2\pi k/n)}{\cos 2\pi k/n} \left(1 - \cos 2\pi k/n
    \right).    
  \end{equation}
  Because $k \leq n/16$ for all $k \in \I$, it is easy to verify that
  $\Delta_k(\theta_i) \leq 8 R \sin^2 (\pi k/n) \leq 8 R \pi^2
  k^2/n^2$. Taking $f_k(\theta_i) = 8 R \pi^2k^2/n^2$ in Corollary
  \ref{arbe}, we obtain that $k_*(i) \geq c (n \sigma^2/R)^{2/5}$ for
  a constant $c$. Also since the function 
  $1-\cos(2x)/\cos(x)$ is  a strongly convex function on
  $[-\pi/4,\pi/4]$ with second derivative  lower bounded by 3, we have
  \begin{equation*}
    \Delta_k(\theta_i) = \frac{R}{k+1} \sum_{j=0}^{k} \left(1 -
      \frac{\cos \frac{4\pi j}{n}}{\cos \frac{2\pi j}{n}} \right) 
      \geq \frac{R}{k+1} \sum_{j=0}^{k}\frac{3}{2}\left(\frac{2\pi j}{n}
      \right)^2 = \frac{R\pi^2 k(2k+1)}{n^2}.    
  \end{equation*}
This gives $k_*(i) \leq C (n \sigma^2/R)^{2/5}$ as well for a constant $C$. We thus have $k_*(i) \asymp (n \sigma^2/R)^{2/5}$ for every $i$. Theorem \ref{rbe} then gives 
  \begin{equation}\label{upco}
    \E_{K^*} \left(\hat{h}_i - h_{K^*}(\theta_i) \right)^2 \leq C
    \left(\frac{\sigma^2\sqrt{R}}{n} \right)^{4/5} \qt{for every $i
      \in \{1, \dots, n\}$}.  
  \end{equation}
Theorem \ref{nt} and inequality \eqref{nti.1} prove that the set estimators $\hat{K}$ and $\hat{K}'$ also converge to $K^*$ at the $n^{-4/5}$ rate. 
}
\end{example}

In the preceding examples, we saw that the optimal rate
$\sigma^2/(k_*(i) + 1)$ for estimating $h_{K^*}(\theta_i)$ did not
depend on $i$. Next, we consider \textit{asymmetric} examples where
the rate changes with $i$.  

\begin{example}[Segment]\label{segm}{\rm
  Suppose $K^*$ is the vertical line segment joining the two points
  $(0, R)$ and $(0, -R)$ for a fixed $R > 0$. One then gets
  $h_{K^*}(\theta) = R |\sin \theta|$ for all $\theta$. For
  simplicity, assume that $n$ is even and consider $i = n/2$ so that
  $\theta_{n/2} = 0$. It can then be verified that 
  \begin{equation*}
    \Delta_k(\theta_{n/2}) = \Delta_k(0) = \frac{R}{k+1}
    \sum_{j=0}^{k} \tan \frac{2 \pi j}{n} \qt{for every $k \in \I$}. 
  \end{equation*}
  Because $j \mapsto \tan (2 \pi j/n)$ is increasing, we get 
  \begin{equation*}
  \frac{3\pi Rk}{8n} \leq \frac{R}{k+1} (\frac{3k}{4}+1)\tan (2 \pi k/4n) 
  \leq \Delta_k(0) \leq R \tan (2 \pi k/n) \leq 2 R \sin (2 \pi k /n)
    \leq \frac{4 \pi R k}{n}. 
  \end{equation*}
  Corollary \ref{arbe} then gives 
  \begin{equation}\label{segm.eq}
    \frac{\sigma^2}{k_*(n/2) + 1} \asymp \left(\frac{\sigma^2 R}{n}
    \right)^{2/3}. 
  \end{equation}
  It was shown in Corollary \ref{theorem::RadiusRBall} that the right
  hand side above represents the maximum possible value of
  $\sigma^2/(k_*(i) + 1)$ when $K^*$ lies in a closed ball of radius
  $R$. Therefore this example presents the situation where estimation
  of $h_{K^*}(\theta_i)$ is the most difficult. See Remark \ref{remco}
  for the connection to smoothness of $h_{K^*}(\cdot)$ at $\theta_i$. 

  Now suppose that $i = 3n/4$ (assume that $n/4$ is an integer for
  simplicity) so that $\theta_i = \pi/2$. Observe then that
  $h_{K^*}(\theta) = R \sin \theta$ (without the modulus) for $\theta
  = \theta_i \pm 4 j \pi /n$ for every $0 \leq j \leq k, k \in
  \I$. Using \eqref{alex}, we have $\Delta_k(\theta_i) = 0$ for every $k
  \in \I$. This immediately gives $k_*(i) = \lfloor n/16
  \rfloor$ and hence 
  \begin{equation}\label{sn}
    \frac{\sigma^2}{k_*(3n/4) + 1} \asymp \frac{\sigma^2}{n}. 
  \end{equation} 
  In this example, the risk for estimating $h_{K^*}(\theta_i)$ changes
  with $i$. For $i = n/2$, we get the $n^{-2/3}$ rate 
  while for $i = 3n/4$, we get the parametric rate. For other values
  of $i$, one gets a range of rates between $n^{-2/3}$ and $n^{-1}$. 

  Because $K^*$ is a polytope with 2 vertices, Theorem
  \ref{theorem::General} and inequality \eqref{nti.2} imply that the
  set estimators $\hat{K}$ and $\hat{K}'$ converge at the near
  parametric rate $\sigma^2\log n/n$. It is interesting to note
  here that even though for some $\theta_i$, the optimal rate of
  estimation of $h_{K^*}(\theta_i)$ is $n^{-2/3}$, the entire set can
  be estimated at the near parametric rate. 
}
\end{example}

\begin{example}[Half-ball]{\rm
  Suppose $K^* := \{(x_1, x_2) : x_1^2 + x_2^2 \leq 1, x_2 \leq
  0\}$. One then has $h_K(\theta) = 1$ for $-\pi \leq \theta \leq 0$
  and $h_K(\theta) = |\cos \theta|$ for $0 < \theta \leq \pi$. Assume
  $n$ is even and take $i = n/2$ so that $\theta_i = 0$. Then
\begin{equation*}
  \Delta_k(0) = \frac{1}{k+1} \sum_{j=0}^{k}
  \left(\frac{\cos 4\pi j/n + 1}{2} - \frac{\cos 4\pi j/n}{\cos 2\pi
      j/n} \frac{\cos 2\pi j/n + 1}{2} \right) = \frac{1}{2(k+1)} 
  \sum_{j=0}^{k} \left(1 - \frac{\cos 4\pi j/n}{\cos 2\pi j/n} \right). 
\end{equation*}
This is exactly as in \eqref{bca} with $R = 1$ and an additional factor
of $1/2$. Arguing as in Example \ref{balle}, we obtain that 
\begin{equation*}
  \frac{\sigma^2}{k_*(n/2) + 1} \asymp \left(\frac{\sigma^2}{n}
  \right)^{4/5}.  
\end{equation*}
Now take $i = 3n/4$ (assume $n/4$ is an integer) so that $\theta_i
= \pi/2$. Observe then that $h_{K^*}(\theta) = |\cos \theta|$ for
$\theta = \theta_i \pm 4 j \pi/n$ for every $0 \leq j \leq k, k \in
\I$. The situation is therefore similar to \eqref{segm.eq} and we
obtain 
\begin{equation*}
  \frac{\sigma^2}{k_*(3n/4) + 1} \asymp \left(\frac{\sigma^2}{n}
  \right)^{2/3}. 
\end{equation*}
Similar to the previous example, the risk for estimating
$h_{K^*}(\theta_i)$ changes with $i$ and varies 
from $n^{-2/3}$ to $n^{-4/5}$. On the other hand, Theorem \ref{nt}
states that the set estimator $\hat{K}$ still estimates $K^*$ at the
rate $n^{-4/5}$.   
}
\end{example}

\begin{remark}[Connection between risk and smoothness]\label{remco}{\rm
The reader may observe that the support functions \eqref{sinsp} and
\eqref{ballsp} in the two examples above differ only by the constant
$R$. It might then seem strange that only the addition of a non-zero
constant changes the risk of estimating $h_{K^*}(\theta_i)$ from
$n^{-1}$ to $n^{-4/5}$. It turns out that the function \eqref{sinsp}
is much more smoother than the function \eqref{ballsp}; the right way
to view smoothness of $h_{K^*}(\cdot)$ is to regard it as a function on
$\R^2$. This is done in the following way. Define, for each $z = (z_1,
z_2) \in \R^2$,  
\begin{equation*}
  h_{K^*}(z) = \max_{(x_1, x_2) \in K^*} \left(x_1 z_1 + x_2 z_2
  \right). 
\end{equation*}
When $z = (\cos \theta, \sin \theta)$ for some $\theta \in \R$, this
definition coincides with our definition of $h_{K^*}(\theta)$. A
standard result (see for example Corollary 1.7.3 and Theorem 1.7.4 in
\cite{Schneider}) states that the subdifferential of $z
\mapsto h_{K^*}(z)$ exists at every $z = (z_1, z_2) \in \R^2$ and is
given by  
\begin{equation*}
  F(K^*, z) := \left\{(x_1, x_2) \in K^*: h_{K^*}(z) = x_1 z_1 + x_2
    z_2 \right\}. 
\end{equation*}
In particular, $z \mapsto h_{K^*}(z)$ is differentiable at $z$ if and
only if $F(K^*, z)$ is a singleton. 

This point of view of studying $h_{K^*}$ as a function on $\R^2$ sheds
qualitative light on the risk bounds obtained in the examples. In the
case of Example \ref{sp} when $K^* = \{(x_1, x_2)\}$, it is 
clear that $F(K^*, z) = \{(x_1, x_2)\}$ for all $z$. Because this set
does not change with $z$, this provides the case
of maximum smoothness (because the derivative is constant) and thus we
get the $n^{-1}$ rate. 

In Example \ref{balle} when $K^*$ is a ball centered at $x = (x_1,
x_2)$ with radius $R$, it can be checked that $F(K^*, z) = \{x + R
z/\|z\|\}$ for every $z \neq 0$. Since $F(K^*, z)$ is a singleton for
each $z \neq 0$, it follows that $z \mapsto h_{K^*}(z)$ is
differentiable for every $z$. For $R \neq 0$, the set $F(K^*, z)$
changes with $z$ and thus here $h_{K^*}$ is not as smooth as in
Example \ref{sp}. This explains the slower rate in Example \ref{balle}
compared to \ref{sp}. 

Finally in Example \ref{segm}, when $K^*$ is the vertical segment
joining $(0, R)$ and $(0, -R)$, it is easy to see that $F(K^*, z) =
K^*$ when $z = (1, 0)$. Here $F(K^*, z)$ is not a singleton which
implies that $h_{K^*}(z)$ is non-differentiable at $z = (1, 0)$. This
is why one gets the slow rate $n^{-2/3}$ for estimating
$h_{K^*}(\theta_{n/2})$ in Example \ref{segm}. 
}
\end{remark}



\section{Discussions}
\label{discussion.sec}

In this paper we study the problems of estimating both the support function at a point, $h_{K^*}(\theta_i)$, and the convex set $K^*$. Data-driven adaptive estimators are constructed and their optimality is established.
For pointwise estimation, the quantity $k_*(i)$, which appears in both the upper bound \eqref{rbe.eq} and the lower bound \eqref{lobo.eq},  is related to the smoothness of $h_{K^*}(\theta)$ at $\theta = \theta_i$. The construction of $\hat{h}_i$ is based on local smoothing together with an optimization scheme for choosing the bandwidth. Smoothing methods for estimating the support function have previously been studied by~\cite{FisherHallTurlachWatson}. Specifically, working under certain smoothness assumptions on the true support function $h_{K^*}(\theta)$, \citet{FisherHallTurlachWatson} estimated it using periodic versions of standard nonparametric regression techniques such
as local regression, kernel smoothing and splines. They evade the problem of bandwidth selection however by assuming that the
true support function is sufficiently smooth. Our estimator comes with a scheme for choosing the bandwidth automatically from the data and hence we do not need any smoothness assumptions on the true convex set. 

To avoid complications, we have assumed throughout the paper that the noise level $\sigma$ is known. In practice, $\sigma$ is typically unknown and needs to be estimated. Under the setting of the present paper, $\sigma$ is easily estimable by using the median of the consecutive differences. Let $\delta_i=Y_{2i} - Y_{2i-1}, \; i = 1, \dots, \lfloor {n\over 2}\rfloor$. A simple robust estimator of the noise level $\sigma$ is the following median absolute deviation (MAD) estimator:
\[
\hat \sigma = \frac{{\rm median} |\delta_i - {\rm median}(\delta_i) |}{1.349}.
\]

It was noted that the construction of our estimators $\hat{K}$ and $\hat{K}'$ given in Section \ref{setes} does not involve any special treatment for polytopes; yet we obtain faster rates for polytopes.
Such automatic adaptation to polytopes has been observed in other contexts: isotonic regression where one gets automatic adaptation for piecewise constant monotone functions (see \citet{GuntuAnnIso}) and convex regression where one gets automatic adaptation for piecewise affine convex functions (see \citet{GSvex}).    
 
Finally, we note that because $\sigma^2/(k_*(i) +1)$ gives the optimal
rate in pointwise estimation, it can potentially be used as a
benchmark to evaluate other estimators for $h_{K^*}(\theta_i)$ such as
the least squares estimator $h_{\hat{K}_{\text{ls}}}(\theta_i)$. This
however is beyond the scope of the current paper.


\section{Proofs of the main results}
\label{pmr}

We prove the main results in this section. Additional technical
results and proofs are given in Appendix \ref{apap}.  
\subsection{Proof of Theorem \ref{rbe}}
\label{pf.rbe}

We provide the proof of Theorem \ref{rbe} here. The proof uses  three simple lemmas: Lemma \ref{keydrop}, \ref{delkstar} and \ref{varcal} which are stated and proved in Appendix \ref{apap}.

Fix $i = 1, \dots, n$. Because $\hat{h}_i = \hat{U}_{\hat{k}(i)}(\theta_i)$, we write
  \begin{equation*}
\left(\hat{h}_i - h_{K^*}(\theta_i) \right)^2 = \sum_{k \in \I}
\left(\hat{U}_{k}(\theta_i) - h_{K^*}(\theta_i) \right)^2
I\left\{\hat{k}(i) = k \right\}
  \end{equation*}
where $I(\cdot)$ denotes the indicator function. Taking expectations
on both sides and using Cauchy-Schwartz inequality, we obtain
\begin{equation*}
  \E_{K^*} \left(\hat{h}_i - h_{K^*}(\theta_i) \right)^2 \leq \sum_{k \in \I}
  \sqrt{\E (\hat{U}_{k}(\theta_i) - h_{K^*}(\theta_i))^4} \sqrt{\P_{K^*}
    \left\{\hat{k}(i) = k \right\}}. 
\end{equation*}
The random variable $\hat{U}_k - h_{K^*}(0)$ is normally distributed and
we know that $\E Z^4 \leq 3 (\E Z^2)^2$ for every gaussian random
variable $Z$. We therefore have
\begin{equation*}
\E_{K^*} \left(\hat{h}_i - h_{K^*}(\theta_i) \right)^2  \leq \sqrt{3}
\sum_{k \in \I}  \E (\hat{U}_{k}(\theta_i) - h_{K^*}(\theta_i))^2 \sqrt{\P_{K^*}
  \left\{\hat{k}(i) = k \right\}}.   
\end{equation*}
Because $\E_{K^*} \hat{U}_k(\theta_i) = U_k(\theta_i)$ (defined in
\eqref{kiest}), we have  
\begin{equation*}
  \E_{K^*} (\hat{U}_{k}(\theta_i) - h_{K^*}(\theta_i))^2 = (U_k(\theta_i) -
  h_{K^*}(\theta_i))^2 + \text{var}(\hat{U}_k(\theta_i)) .   
\end{equation*}
Because $L_k(\theta_i) \leq h_{K^*}(\theta_i) \leq U_k(\theta_i)$, it
is clear that $U_k(\theta_i) - h_{K^*}(\theta_i) \leq U_k(\theta) -
L_k(\theta_i) = \Delta_k(\theta_i)$. Also, Lemma \ref{varcal} states
that the variance of $\hat{U}_k$ is at most $\sigma^2/(k+1)$. Putting
these together, we obtain 
\begin{equation*}
    \E_{K^*} \left(\hat{h}_{i} - h_{K^*}(\theta_i) \right)^2 \leq
    \sqrt{3} \sum_{k \in \I} \left(\Delta_k^2(\theta_i) +
      \frac{\sigma^2}{k+1} \right) \sqrt{\P_{K^*}\left\{\hat{k}(i) = k
      \right\}}.  
\end{equation*} 
The proof of \eqref{rbe.eq} will therefore be complete if we show that 
\begin{equation}\label{pts}
\sum_{k \in \I} \left(\Delta_k^2(\theta_i) + \frac{\sigma^2}{k+1}
\right) \sqrt{\P_{K^*} \left\{\hat{k}(i) = k \right\}} \leq C
    \frac{\sigma^2}{k_*(i)+1}
\end{equation} 
for a universal positive constant $C$. 

Below, we write $\Delta_k, \hat{k}$ and $k_*$ for $\Delta_k(\theta_i),
\hat{k}(i)$ and $k_*(i)$ respectively for ease of notation. We also
write $\P$ for $\P_{K^*}$. 

We prove \eqref{pts} by considering the two cases: $k \leq k_*, k \in
\I$ and $k > k_*, k \in \I$ separately.     

The first case is $k \leq k_*, k \in \I$. By Lemma \ref{keydrop}  and
\eqref{delkstar.eq}, we get  
\begin{equation*}
  \Delta_k \leq \Delta_{k_*} \leq  \frac{6 (\sqrt{2} - 1)\sigma}{\sqrt{k_*+1}} 
  \leq \frac{6(\sqrt{2} - 1)\sigma}{\sqrt{k+1}}
\end{equation*}
and consequently 
\begin{equation}\label{b1}
  \Delta_k^2 + \frac{\sigma^2}{k+1} \leq \frac{\sigma^2}{k+1}
  \left(36(\sqrt{2} - 1)^2 + 1\right) \qt{for all $k \leq k_*, k \in
    \I$}. 
\end{equation}
We bound $\P \{ \hat{k} = k\}$ by writing 
\begin{equation*}
  \P \{ \hat{k} = k \} \leq \P \left\{ \left(\hat{\Delta}_k \right)^+
    + \frac{2\sigma }{\sqrt{k+1}} \leq \left(\hat{\Delta}_{k_*} \right)^+
    + \frac{2\sigma }{\sqrt{k_*+1}}  \right\} \leq \P \left\{
    \left(\hat{\Delta}_{k_*} \right)^+ \geq \frac{2\sigma}{\sqrt{k+1}} - 
    \frac{2\sigma }{\sqrt{k_*+1}}
  \right\}. 
\end{equation*}
Because $k \leq k_*$, the positive part above can be dropped and we obtain
\begin{equation*}
  \P \{ \hat{k} = k \} \leq \P \left\{
    \hat{\Delta}_{k_*} \geq \frac{2\sigma}{\sqrt{k+1}} - \frac{2\sigma}{\sqrt{k_*+1}}
  \right\}.  
\end{equation*}
Because $\hat{\Delta}_{k_*}$ is normally distributed with mean
$\Delta_{k_*}$, we have
\begin{equation*}
  \P \{ \hat{k} = k \} \leq \P \left\{Z \geq
    \frac{2\sigma (k+1)^{-1/2} - 2\sigma  (k_*+1)^{-1/2} -
      \Delta_{k_*}}{\sqrt{\text{var}(\hat{\Delta}_{k_*})}} \right\},
\end{equation*}
where $Z$ is a standard normal random variable. From
\eqref{delkstar.eq}, we have 
\begin{equation*}
  \frac{2\sigma}{\sqrt{k+1} } - \frac{2\sigma}{\sqrt{k_*+1}} -
      \Delta_{k_*} \geq \frac{2\sigma }{\sqrt{k+1}} \left(1 -
        \sqrt{\frac{k+1}{k_*+1}} \left(3\sqrt{2} - 2 \right) \right).
\end{equation*}
As a result, 
\begin{equation*}
  \P \{ \hat{k} = k \} \leq \P \left\{Z \geq \frac{2\sigma
      }{\sqrt{(k+1) \text{var}(\hat{\Delta}_{k_*})}} \left(1 -
      \sqrt{\frac{k+1}{k_*+1}} \left(3\sqrt{2} - 2 \right) \right) 
         \right\}. 
\end{equation*}
Suppose 
\begin{equation*}
  \tilde{k} := (k_*+1) \left(3\sqrt{2} - 2 \right)^{-2}-1. 
\end{equation*}
For $k < \tilde{k}$, we use the bound given by Lemma \ref{varcal} on
the variance of $\hat{\Delta}_{k^*}$ to obtain 
\begin{equation*}
  \P \{ \hat{k} = k \} \leq \P \left\{Z \geq
    2\left(\sqrt{\frac{k_*+1}{k+1}} - 3\sqrt{2}  + 2 \right) \right\}
  \leq \exp \left(- 2\left[\sqrt{\frac{k_*+1}{k+1}} - 3\sqrt{2} +
      2\right]^2 \right).  
\end{equation*} 
Using this and \eqref{b1}, we see that the quantity 
\begin{equation*}
\sum_{k < \tilde{k}, k \in \I} \left(\Delta_k^2 + \frac{\sigma^2}{k+1}
\right) \sqrt{\P \{\hat{k} = k \}}   
\end{equation*}
is bounded from above by 
\begin{equation*}
 \frac{\sigma^2}{k_* + 1} \left(36(\sqrt{2} - 1)^2 + 1 \right) 
\sum_{k < \tilde{k}, k \in \I}  \frac{k_*+1}{k+1} \exp \left(-
    \left[\sqrt{\frac{k_*+1}{k+1}}  - 3\sqrt{2}  + 2    \right]^2
  \right).  
\end{equation*}
Because $\I$ consists of integers of the form $2^j$, it follows that
for any two successive integers $k_1$ and $k_2$ in $\I$, we have $3/2
\leq (k_1+1)/(k_2 + 1) \leq 2$. Using this, it is easily seen that 
\begin{equation*}
  \sum_{k < \tilde{k}, k \in \I}    \frac{k_*+1}{k+1} \exp \left(-
   \left[\sqrt{\frac{k_*+1}{k+1}} - 3\sqrt{2} + 2 \right]^2  \right) 
\end{equation*}
is bounded from above by 
\begin{equation*}
  \sum_{j \geq 4}   2^j \exp \left(-\left[(3/2)^{j/2} - 3\sqrt{2} +
      2\right]^2  \right) + \sum_{0\leq j\leq 3} 2^j, 
\end{equation*}
which is just a universal positive constant. We have proved therefore
that 
\begin{equation}\label{sp11}
\sum_{k < \tilde{k}, k \in \I}\left(\Delta_k^2 + \frac{\sigma^2}{k+1}
\right) \sqrt{\P \{\hat{k} = k \}} \leq \frac{C_1 \sigma^2}{k_*+1},
\end{equation}
for a positive constant $C_1$.

For $\tilde{k} \leq k \leq k_*$, we simply use \eqref{b1} along with
the trivial bound $\P \{\hat{k} = k \} \leq 1$ to get
\begin{equation*}
\sum_{\tilde{k} \leq k \leq k_*, k \in \I} \left(\Delta_k^2 +
  \frac{\sigma^2}{k+1}  \right) \sqrt{\P \{\hat{k} = k \}}  \leq
\left(36(\sqrt{2} - 1)^2 + 1 \right) \frac{\sigma^2}{k_*+1}
\sum_{\tilde{k} \leq k < k_*, k \in \I} \frac{k_*+1}{k+1}.  
\end{equation*}
Once again because $\I$ consists of integers of the form $2^j$, we get 
\begin{equation*}
  \sum_{\tilde{k} \leq k \leq k_*, k \in \I} \frac{k_*+1}{k+1} \leq \sum_{j
    \geq 0} 2^j \left\{(3/2)^j \leq \left(3\sqrt{2} - 2\right)^2 \right\}.
\end{equation*}
The right hand side above is just a constant. It follows therefore that 
\begin{equation}\label{sp12}
\sum_{\tilde{k} \leq k \leq k_*, k \in \I}\left(\Delta_k^2 + \frac{\sigma^2}{k+1}
\right) \sqrt{\P \{\hat{k} = k \}} \leq \frac{C_2\sigma^2}{k_*+1},
\end{equation}
for a positive constant $C_2$. Combining~\eqref{sp11}
and~\eqref{sp12}, we deduce that 
\begin{equation}\label{sp1}
  \sum_{k \leq k_*, k \in \I}\left(\Delta_k^2 + \frac{\sigma^2}{k+1}
\right) \sqrt{\P \{\hat{k} = k \}} \leq \frac{C\sigma^2}{k_*+1}
\end{equation}
where $C := C_1 + C_2$ is a universal positive constant. 

We next deal with the case $k > k_*, k \in \I$. Assume that $\{k \in
\I: k > k_* \}$ is non-empty for otherwise there is nothing to
prove. By the first part of \eqref{an.delkstar.eq}, we get 
\begin{equation}\label{simplifySum}
  \sum_{k >  k_*, k \in \I} \left(\Delta_k^2 + \frac{\sigma^2}{k+1}
\right) \sqrt{\P \{\hat{k} = k \}} \leq \left(1 + \frac{1}{(\sqrt{6} - 2)^2}
 \right)   \sum_{k > k_*, k \in \I} \Delta_k^2 \sqrt{\P \{\hat{k} = k \}}. 
\end{equation}
We first bound $\P\{\hat{k} = k\}$ for $k > k_*, k \in \I$. We proceed by
writing 
\begin{align*}
  \P \{\hat{k} = k \} & \leq \P \left\{\hat{\Delta}_k^+ +
    \frac{ 2\sigma}{\sqrt{k+1}} \leq \hat{\Delta}_{k_*}^+ +
    \frac{ 2\sigma}{\sqrt{k_*+1}} \right\} \\
&\leq \P \left\{\hat{\Delta}_k +
    \frac{ 2\sigma}{\sqrt{k+1}} \leq \hat{\Delta}_{k_*}^+ +
    \frac{ 2\sigma}{\sqrt{k_*+1}} \right\}  \qt{(because $x \leq
    x^+$)} \\ 
&\leq \P \left\{\hat{\Delta}_k +
    \frac{ 2\sigma}{\sqrt{k+1}} \leq \hat{\Delta}_{k_*} + \frac{
      2\sigma}{\sqrt{k_*+1}} \right\} + \P_K \left\{\hat{\Delta}_k +
    \frac{ 2\sigma}{\sqrt{k+1}} \leq \frac{
      2\sigma}{\sqrt{k_*+1}} \right\} \\
&\leq \P \left\{\hat{\Delta}_k \leq \hat{\Delta}_{k_*} + \frac{
      2\sigma}{\sqrt{k_*+1}} \right\} + \P_K \left\{\hat{\Delta}_k \leq
    \frac{ 2\sigma}{\sqrt{k_*+1}} \right\} \\
&\leq \P \left\{\hat{\Delta}_{k_*} - \hat{\Delta}_k \geq - \frac{
    2\sigma}{\sqrt{k_*+1}}\right\} + \P \left\{- \hat{\Delta}_k 
\geq - \frac{ 2\sigma}{\sqrt{k_*+1}}\right\} 
\end{align*}
Both $\hat{\Delta}_{k_*} - \hat{\Delta}_k$ and $\hat{\Delta}_k$ are
normally distributed with means $\Delta_{k_*} - \Delta_k$ and
$\Delta_k$ respectively. As a result
\begin{equation*}
\P \{\hat{k} = k \} \leq \P \left\{Z \geq \frac{\Delta_k - \Delta_{k_*} - 
    2\sigma (k_*+1)^{-1/2}}{\sqrt{\text{var}(\hat{\Delta}_{k_*} -
      \hat{\Delta}_k)}} \right\} + \P \left\{Z \geq \frac{\Delta_k - 
    2\sigma (k_*+1)^{-1/2}}{\sqrt{\text{var}(\hat{\Delta}_k)}}
\right\}
\end{equation*}
where $Z$ is a standard normal random variable. Using
\eqref{delkstar.eq}, we obtain 
\begin{equation*}
\P \{\hat{k} = k \} \leq \P \left\{Z \geq \frac{\Delta_k - 2\sigma
    (k_*+1)^{-1/2} \left(3\sqrt{2} - 2\right)}{\sqrt{\text{var}(\hat{\Delta}_{k_*} - 
      \hat{\Delta}_k)}} \right\} + \P \left\{Z \geq \frac{\Delta_k - 
     2\sigma (k_*+1)^{-1/2}}{\sqrt{\text{var}(\hat{\Delta}_k)}}
\right\}.  
\end{equation*}
By the Cauchy-Schwarz inequality and Lemma~\ref{varcal}, we get, for
$k > k_*$, 
\begin{equation*}
  \sqrt{\text{var}(\hat{\Delta}_{k_*} - \hat{\Delta}_k)} \leq
  \sqrt{\text{var}(\hat{\Delta}_{k_*})} +
  \sqrt{\text{var}(\hat{\Delta}_k)} \leq \frac{\sigma}{\sqrt{k+1}} +
  \frac{\sigma}{\sqrt{k_*+1}} \leq \frac{2 \sigma}{\sqrt{k_*+1}} 
\end{equation*}
Also $\text{var}(\hat{\Delta}_k) \leq \sigma^2/(k+1) \leq
\sigma^2/(k_*+1)$. Therefore if $k > k_*, k \in \I$ is such that
\begin{equation}\label{nask}
  \Delta_k \geq 2\sigma (k_*+1)^{-1/2} \left(3\sqrt{2} - 2 \right),
\end{equation}
we obtain
\begin{align*}
  \P \{\hat{k} = k \} &\leq \P \left\{Z \geq \frac{\Delta_k - 2\sigma
    (k_*+1)^{-1/2} \left(3\sqrt{2} - 2 \right)}{\sigma \sqrt{2}
    (k_*+1)^{-1/2}} \right\} + \P \left\{Z \geq \frac{\Delta_k -   
     2\sigma (k_*+1)^{-1/2}}{\sigma (k_*+1)^{-1/2}} \right\} \\
&\leq  2 \P \left\{Z \geq \frac{\Delta_k - 2\sigma
    (k_*+1)^{-1/2} \left(3\sqrt{2} - 2 \right)}{\sigma \sqrt{2}
    (k_*+1)^{-1/2}} \right\}   \\
&\leq 2 \exp \left(-\frac{k_*+1}{2\sigma^2}
  \left(\Delta_k - 2\sigma (k_*+1)^{-1/2} (3\sqrt{2} - 2) \right)^2
\right). 
\end{align*}
Using the inequality $(x-y)^2 \geq x^2/2 - y^2$ with $x = \Delta_k$
and $y = 2\sigma (k_*+1)^{-1/2} (3\sqrt{2} - 2)$, we obtain
\begin{equation}\label{desh}
  \P \{\hat{k} = k \} \leq 2 \exp \left(2(3 \sqrt{2} - 2)^2 \right)
  \exp \left(- \frac{(k_*+1) \Delta_k^2}{4 \sigma^2}
  \right) 
\end{equation}
whenever $k \in I, k > k_*$ satisfies \eqref{nask}. It is easy
to see that when~\eqref{nask} is not satisfied, the right hand side
above is larger than 2. Thus, inequality~\eqref{desh} is true for all
$k \in \I, k > k_*$. As a result, 
\begin{equation}\label{kej}
  \Delta_k^2 \sqrt{\P \{\hat{k} = k \}} \leq \sqrt{2} \exp \left((3
    \sqrt{2} - 2)^2 \right) \xi \left( \Delta_k^2\right) \qt{for all
    $k \in \I, k > k_*$}. 
\end{equation}
where 
\begin{equation*}
  \xi(z) := z \exp \left(-\frac{(k_*+1) z}{8 \sigma^2} \right) \qt{for
  $z > 0$}. 
\end{equation*}
By \eqref{simplifySum} and \eqref{kej}, the proof would therefore be
complete if we show that $\sum_{k \in \I: k > k_*} \xi
\left(\Delta_k^2 \right)$ is bounded from above by a universal
positive constant. For this, note first that the function 
$\xi(z)$ is decreasing for $z \geq \breve{z} := 8 \sigma^2/(k_* + 1)$
and attains its maximum over $z > 0$ at $z = \breve{z}$. Note also the
second part of inequality \eqref{an.delkstar.eq} gives $\Delta_k^2 \geq z_k$ 
for all $k \in \I, k > k_*$ where 
\begin{equation*}
  z_k := \frac{(\sqrt{6} - 2)^2 \sigma^2 (k+1)}{4(k_* + 1)^2}
\end{equation*}
We therefore get 
\begin{align*}
  \xi \left(\Delta_k^2 \right) \leq \xi(\max(z_k, \breve{z})) &= \max(z_k,
  \breve{z}) \exp \left(\frac{-(k_*+1)  \max(z_k, \breve{z})}{8 \sigma^2} \right) \\
&\leq \max(z_k, \breve{z}) \exp \left(\frac{-(k_*+1) z_k }{8 \sigma^2}
\right) \leq (z_k + \breve{z}) \exp \left(\frac{-(k_*+1) z_k }{8
    \sigma^2} \right).   
\end{align*}
Because $k > k_*$, it is easy to see that 
\begin{equation*}
  \breve{z} = \frac{8 \sigma^2}{k_* + 1} \leq \frac{8 \sigma^2
    (k+1)}{(k_* + 1)^2}. 
\end{equation*}
We deduce that
\begin{equation*}
  \xi \left(\Delta_k^2 \right) \leq \left[\frac{(\sqrt{6} - 2)^2}{4} +
  8\right] \frac{\sigma^2 (k+1)}{(k_*+1)^2} \exp \left(-
  \frac{(\sqrt{6} - 2)^2}{32} \frac{k+1}{k_*+1} \right). 
\end{equation*}
Denoting the constants above by $c_1$ and $c_2$, we can write 
\begin{equation*}
  \sum_{k \in \I: k > k_*} \xi \left( \Delta_k^2 \right) \leq
  \frac{c_1 \sigma^2}{k_* + 1} \sum_{k \in \I: k > k_*} \frac{k+1}{k_*
  + 1} \exp \left(- \frac{k+1}{c_2(k_* + 1)} \right).  
\end{equation*}
The sum in the right hand side above is easily seen to be bounded from
above by 
\begin{equation*}
  \sum_{j \geq 0} 2^j \exp \left(-\frac{1}{c_2} \left(\frac{3}{2}
    \right)^j \right)
\end{equation*}
which is further bounded by a universal constant. This completes the
proof of Theorem \ref{rbe}. 

\subsection{Proof of Theorem \ref{lobo}}\label{pf.lobo}
This subsection is dedicated to the proof of Theorem \ref{lobo}. We
use Lemma \ref{useaux} which is stated and proved in Section
\ref{apap}. We also use a classical inequality due to
\citet{LeCam:86book} which states that for every estimator $\tilde{h}$
and compact, convex set $L^*$,      
\begin{equation}\label{lecam}
  \max \left[\E_{K^*} \left(\tilde{h} - h_{K^*}(\theta_i) \right)^2,
    \E_{L^*} \left(\tilde{h} - h_{L^*}(\theta_i) \right)^2 \right]
  \geq \frac{1}{4}  \left(h_{K^*}(\theta_i) - h_{L^*}(\theta_i)
  \right)^2 \left(1 - \|P_{K^*} - P_{L^*} \|_{TV} \right). 
\end{equation}
Here $P_{L^*}$ is the product of the Gaussian probability
measures with mean $h_{L^*}(\theta_i)$ and variance $\sigma^2$
for $i = 1, \dots, n$. Also $\|P-Q\|_{TV}$ denotes the total variation
distance between $P$ and $Q$. 

For ease of notation, we assume, without loss of generality, that
$\theta_i = 0$. We also write $\Delta_k$ for $\Delta_k(\theta_i)$ and
$k_*$ for $k_*(i)$. 

Suppose first that $K^*$ satisfies the following
condition: There exists some $\alpha \in (0, \pi/4)$ such that 
\begin{equation}\label{con1}
  \frac{h_{K^*}(\alpha) + h_{K^*}(-\alpha)}{2 \cos \alpha} -
  h_{K^*}(0) > \frac{\sigma}{\sqrt{n_{\alpha}}} 
\end{equation}
where $n_{\alpha}$ denotes the number of integers $i$ for which
$-\alpha < 2i \pi/n < \alpha$. This condition will not be satisfied,
for example, when $K^*$ is a singleton. We shall handle such $K^*$
later. Observe that $n_{\alpha} \geq 1$ for all $0 < \alpha < \pi/4$
because we can take $i = 0$. 

Let us define, for each $\alpha \in (0, \pi/4)$, 
\begin{equation*}
  a_K^*(\alpha) := \left(\frac{h_{K^*}(\alpha) + h_{K^*}(-\alpha)}{2\cos
      \alpha}, \frac{h_{K^*}(\alpha) - h_{K^*}(-\alpha)}{2 \sin
      \alpha} \right). 
\end{equation*}
and let $L^* = L^*(\alpha)$ be defined as the smallest
convex set that contains both $K^*$ and the point $a_{K^*}(\alpha)$. In
other words, $L^*$ is the convex hull of $K^* \cup
\{a_{K^*}(\alpha)\}$.  

We now use Le Cam's inequality \eqref{lecam}. To control the total
variation distance in the right hand side of \eqref{lecam}, we use
Pinsker's inequality: 
\begin{equation*}
  ||P_{K^*} - P_{L^*}||_{TV} \leq \sqrt{\frac{1}{2}
    D(P_{K^*} || P_{L^*})},
\end{equation*}
and the fact that (note that $\theta_i = 2 \pi i /n - \pi$)
\begin{equation*}
  D(P_{K^*} || P_{L^*}) = \frac{1}{2\sigma^2}
  \sum_{i=1}^n \left(h_{K^*}(2i\pi/n-\pi) - h_{L^*}(2i\pi/n-\pi)
  \right)^2. 
\end{equation*}
The support function of $L^*$ is easily seen to be
the maximum of the support functions of $K^*$ and the singleton
$\{a_{K^*}(\alpha)\}$. Therefore,  
\begin{align*}
  h_{L^*}(\theta) &:= \max \left(h_{K^*}(\theta),
    \frac{h_{K^*}(\alpha) + h_{K^*}(-\alpha)}{2 \cos \alpha} \cos \theta +
    \frac{h_{K^*}(\alpha) - h_{K^*}(-\alpha)}{2 \sin \alpha} \sin \theta
  \right) \\ 
&= \max \left(h_{K^*}(\theta), \frac{\sin(\theta + \alpha)}{\sin 2\alpha}
  h_{K^*}(\alpha) + \frac{\sin (\alpha - \theta)}{\sin 2 \alpha}
  h_{K^*}(-\alpha) \right). 
\end{align*}
Using~\eqref{genvit}, it can be shown that
\begin{equation}\label{pe1}
  h_{K^*}(\theta) \leq \frac{\sin(\theta + \alpha)}{\sin 2\alpha}
  h_{K^*}(\alpha) + \frac{\sin (\alpha - \theta)}{\sin 2 \alpha}
  h_{K^*}(-\alpha) \qt{for $-\alpha < \theta < \alpha$},
\end{equation}
and 
\begin{equation}\label{pe2}
  h_{K^*}(\theta) \geq \frac{\sin(\theta + \alpha)}{\sin 2\alpha}
  h_{K^*}(\alpha) + \frac{\sin (\alpha - \theta)}{\sin 2 \alpha}
  h_{K^*}(-\alpha) \qt{for $\theta \in [-\pi, -\alpha] \cup [\alpha,
    \pi]$}.  
\end{equation}
To see this, assume that $\theta > 0$ without loss of generality. We
then work with the two separate cases $\theta \in [0, \alpha]$ and
$\theta \in [\alpha, \pi]$. In the first case, apply~\eqref{genvit}
with $\alpha_1 = \alpha, \alpha = \theta$ and $\alpha_2 = -\alpha$ to
get~\eqref{pe1}. In the second case, apply~\eqref{genvit} with
$\alpha_1 = \theta, \alpha = \alpha$ and $\alpha_2 = -\alpha$ to
get~\eqref{pe2}.  

As a result of~\eqref{pe1} and~\eqref{pe2}, we get that
\begin{equation*}
  h_{L^*}(\theta) = \frac{\sin(\theta + \alpha)}{\sin 2\alpha}
  h_{K^*}(\alpha) + \frac{\sin (\alpha - \theta)}{\sin 2 \alpha}
  h_{K^*}(-\alpha) \qt{for $-\alpha < \theta < \alpha$},
\end{equation*}
and that $h_{L^*}(\theta)$ equals $h_{K^*}(\theta)$ for
every other $\theta$ in $(-\pi, \pi]$. 

We now give an upper bound on $h_{L^*}(\theta) -
h_{K^*}(\theta)$ for $0 \leq \theta < \alpha$. Using~\eqref{genvit} with
$\alpha_1 = \theta, \alpha = 0$ and $\alpha_2 = -\alpha$, we obtain 
\begin{equation*}
  h_{K^*}(\theta) \geq \frac{\sin (\alpha + \theta)}{\sin \alpha} h_{K^*}(0) -
  \frac{\sin \theta}{\sin \alpha} h_{K^*}(-\alpha). 
\end{equation*}
Thus for $0 \leq \theta < \alpha$, we obtain the inequality
\begin{align*}
 0 \leq h_{L^*}(\theta) - h_{K^*}(\theta) &= \frac{\sin (\theta +
    \alpha)}{\sin 2\alpha} h_{K^*}(\alpha) + \frac{\sin (\alpha -
    \theta)}{\sin 2\alpha} h_{K^*}(-\alpha) - h_{K^*}(\theta) \\
&\leq \frac{\sin (\theta + \alpha)}{\sin \alpha}
\left(\frac{h_{K^*}(\alpha) + h_{K^*}(-\alpha)}{2 \cos \alpha} -
  h_{K^*}(0) \right).  
\end{align*}
Because $0 < \alpha < \pi/4, 0 \leq \theta \leq \alpha$, we use the
fact that the sine function is increasing on $(0, \pi/2)$ to deduce
that
\begin{equation*}
   0 \leq h_{L^*}(\theta) - h_{K^*}(\theta) \leq
   \frac{h_{K^*}(\alpha) + h_{K^*}(-\alpha)}{2 \cos \alpha} -
   h_{K^*}(0) \qt{for all $0 \leq \theta < \alpha$}. 
\end{equation*}
One can similarly deduce the same inequality for the case $-\alpha <
\theta \leq 0$ as well. 

Because of this and the fact that $h_{L^*}(\theta)$ equals
$h_{K^*}(\theta)$ for all $\theta$ in $(-\pi, \pi]$ that are not in
the interval $(-\alpha, \alpha)$, we obtain 
\begin{align*}
  D(P_{K^*} || P_{L^*}) &= \frac{1}{2\sigma^2}
  \sum_{i=1}^n \left(h_{K^*}(2i\pi/n-\pi) - h_{L^*}(2i\pi/n-\pi)
  \right)^2 \\
&\leq \frac{n_{\alpha}}{2 \sigma^2}
  \left(\frac{h_{K^*}(\alpha) + h_{K^*}(-\alpha)}{2 \cos \alpha} -
    h_{K^*}(0) \right)^2. 
\end{align*}
Also because $h_{L^*}(0) = (h_{K^*}(\alpha) + h_{K^*}(-\alpha))/(2\cos
    \alpha)$, we obtain, by \eqref{lecam}, that 
\begin{equation}\label{lbmaineq}
  r \geq \frac{1}{4} \left(\frac{h_{K^*}(\alpha)
      + h_{K^*}(-\alpha)}{2 \cos \alpha} - h_{K^*}(0) \right)^2 \left(1 -
    \sqrt{\frac{n_{\alpha}}{4 \sigma^2}} \left(\frac{h_{K^*}(\alpha) +
        h_{K^*}(-\alpha)}{2 \cos \alpha} - h_{K^*}(0) \right) \right)
\end{equation}
for every $0 < \alpha < \pi/4$ where 
\begin{equation}\label{rde}
r := \inf_{\tilde{h}} \max \left[\E_{K^*} \left(\tilde{h} -
    h_{K^*}(\theta_i) \right)^2, \E_{L^*} \left(\tilde{h} -
    h_{L^*}(\theta_i) \right)^2 \right]
\end{equation}
where the infimum above is over all estimators $\tilde{h}$. Let us now
define $\alpha_*$ by  
\begin{equation*}
  \alpha_* := \inf \left\{0 < \alpha < \pi/4 : \frac{h_{K^*}(\alpha) +
      h_{K^*}(-\alpha)}{2 \cos \alpha} - h_{K^*}(0) >
    \frac{\sigma}{\sqrt{n_{\alpha}}} \right\}. 
\end{equation*}
Note first that $\alpha_* > 0$ because $n_{\alpha} \geq 1$ for all
$\alpha$ and thus for $\alpha$ very small while the quantity 
$(h_{K^*}(\alpha) + h_{K^*}(-\alpha))/(2\cos \alpha) - h_{K^*}(0)$
becomes close to 0 for small $\alpha$ (by continuity of
$h_{K^*}(\cdot)$). 

Also because we have assumed \eqref{con1}, it follows that $0 <
\alpha_* < \pi/4$. Now for each $\epsilon > 0$ sufficiently small, we
have   
\begin{equation*}
  \frac{h_{K^*}(\alpha_* - \epsilon) + h_{K^*}(-\alpha_* +
    \epsilon)}{2 \cos (\alpha_* - \epsilon)} - h_{K^*}(0) \leq 
  \frac{\sigma}{\sqrt{n_{\alpha_* - \epsilon}}}.  
\end{equation*}
Letting $\epsilon \downarrow 0$ in the above and using the fact that 
$n_{\alpha_* - \epsilon} \rightarrow n_{\alpha_*}$ and the continuity
of $h_{K^*}$, we deduce 
\begin{equation}\label{ddd}
 \frac{h_{K^*}(\alpha_*) + h_{K^*}(-\alpha_*)}{2 \cos \alpha_*} - h_{K^*}(0) \leq
    \frac{\sigma}{\sqrt{n_{\alpha_*}}}. 
\end{equation}
Because $0 < \alpha_* < \pi/4$, by the definition of the infimum,
there exists a decreasing sequence $\{\alpha_k\} \in (0, \pi/4 )$
converging to $\alpha_*$ such that 
\begin{equation*}
 \frac{h_{K^*}(\alpha_k) +
      h_{K^*}(-\alpha_k)}{2 \cos \alpha_k} - h_{K^*}(0) >
    \frac{\sigma}{\sqrt{n_{\alpha_k}}} \qt{for all $k$}. 
\end{equation*}
For $k$ large, $n_{\alpha_k}$ is either $n_{\alpha_*}$ or
$n_{\alpha_*} + 2$, and hence letting $k \rightarrow \infty$, we get 
\begin{equation*}
 \frac{h_{K^*}(\alpha_*) + h_{K^*}(-\alpha_*)}{2 \cos \alpha_*} - h_{K^*}(0) \geq
    \frac{\sigma}{\sqrt{n_{\alpha_*}+2}} \geq \frac{1}{\sqrt{3}}
    \frac{\sigma}{\sqrt{n_{\alpha_*}}},
\end{equation*}
where we also used that $n_{\alpha_*} \geq 1$. Combining the above
with \eqref{ddd}, we conclude that 
\begin{equation*}
 \frac{1}{\sqrt{3}}
    \frac{\sigma}{\sqrt{n_{\alpha_*}}} \leq \frac{h_{K^*}(\alpha_*) +
      h_{K^*}(-\alpha_*)}{2 \cos \alpha_*} - h_{K^*}(0) \leq
    \frac{\sigma}{\sqrt{n_{\alpha_*}}}.  
\end{equation*}
Using $\alpha = \alpha_*$ in \eqref{lbmaineq}, we get 
\begin{equation}\label{jobdone}
  r \geq \frac{\sigma^2}{24 n_{\alpha_*}}. 
\end{equation}
We shall now show that 
\begin{equation}\label{jjw}
  \alpha_* \leq \tilde{\alpha} := \frac{8 (k_* + 1) \pi}{n} 
\end{equation}
when $8(k_* + 1) \pi/n \leq \pi/4$ (otherwise \eqref{jjw} is
obvious). This would imply, because $\alpha \mapsto n_{\alpha}$ is
non-decreasing, that 
\begin{equation*}
  n_{\alpha_*} \leq n_{\tilde{\alpha}} = \frac{n \tilde{\alpha}}{\pi} -
  1 = 8 k_* + 7. 
\end{equation*}
This and \eqref{jobdone} would give
\begin{equation*}
  r \geq \frac{\sigma^2}{24(8k_* + 7)} \geq \frac{c \sigma^2}{k_* + 1} 
\end{equation*}
for a positive constant $c$. This would prove the theorem when  
assumption \eqref{con1} is true. 

To prove \eqref{jjw}, we only need to show that 
\begin{equation}\label{toos}
  \frac{h_{K^*}(\tilde{\alpha}) + h_{K^*}(-\tilde{\alpha})}{2 \cos
    \tilde{\alpha}} - h_{K^*}(0) >
  \frac{\sigma}{\sqrt{n_{\tilde{\alpha}}}} = \frac{\sigma}{\sqrt{8k_*
      + 7}}. 
\end{equation}
We verify this via Lemma \ref{useaux} on a case-by-case basis. When
$k_* = 0$, we have $\tilde{\alpha} = 8\pi/n$ so that, by Lemma
\ref{useaux}, the left hand side above is bounded from below by
$\Delta_2$. Because $k_*$ is zero, by definition of $k_*$, we have
\begin{equation*}
  \Delta_2 + \frac{2\sigma}{\sqrt{3}} \geq \Delta_0 + 2\sigma = 2
  \sigma. 
\end{equation*}
This gives $\Delta_2 \geq 2 \sigma(1 - (1/\sqrt{3}))$ which can be
verified to be larger than $\sigma/\sqrt{8k_* + 7} = \sigma/\sqrt{7}$.  

When $k_* = 1$, we have $\tilde{\alpha} = 16\pi/n$ so that, by Lemma 
\ref{useaux}, the left hand side in \eqref{toos} is bounded from below by
$\Delta_4$. Because $k_* = 1$, by definition of $k_*$, we have
\begin{equation*}
  \Delta_4 + \frac{2\sigma}{\sqrt{5}} \geq \Delta_1 +
  \frac{2\sigma}{\sqrt{2}} \geq \frac{2\sigma}{\sqrt{2}}
\end{equation*}
which gives $\Delta_4 \geq 2 \sigma((1/\sqrt{2}) -
(1/\sqrt{5}))$. This can be verified to be larger than
$\sigma/\sqrt{8k_* + 7} = \sigma/\sqrt{15}$. 

When $k_* \geq 2$, we again use Lemma \ref{useaux} to argue that the
left hand side in \eqref{toos} is bounded from below by $\Delta_{2(k_*
  + 1)}$. Because $\Delta_k$ is increasing in $k$ (Lemma
\ref{keydrop}), we have $\Delta_{2(k_* + 1)} \geq \Delta_{2k_*}$. By
the definition of $k_*$ (and the fact that $\Delta_{k_*} \geq 0$), we
have  
\begin{equation*}
  \Delta_{2k_*} \geq \frac{2 \sigma}{k_* + 1} \left( 1 -
    \sqrt{\frac{k_* + 1}{2k_* + 1}} \right). 
\end{equation*}
Because $k_* \geq 2$, it can be easily checked that $(k_* + 1)/(2k_* +
1) \leq 3/5$ and $(8k_* + 7)/(k_* + 1) \geq 23/3$. These, together
with the fact that $2(1 - \sqrt{3/5}) \sqrt{23/3} > 1$, imply
\eqref{toos}. This completes the proof of the theorem when 
assumption \eqref{con1} holds. 

We now deal with the simpler case when \eqref{con1} is violated. When
\eqref{con1} is violated, we first show that
\begin{equation}\label{will}
  k_* > \frac{12n }{16(1+2\sqrt{3} )^2}-1. 
\end{equation}
To see this, note first that, because \eqref{con1} is violated, we
have 
\begin{equation*}
  \frac{h_{K^*}(\alpha) + h_{K^*}(-\alpha)}{2\cos \alpha} - h_{K^*}(0) \leq
  \frac{\sigma}{\sqrt{n_{\alpha}}} \leq \sigma \left(\frac{n
      \alpha}{\pi} - 1 \right)^{-1/2}
\end{equation*}
for all $\alpha \in (0, \pi/4]$. Lemma~\ref{useaux} implies that for
every $1\leq k \leq n/16$, we get
\begin{equation*}
  \Delta_k \leq \frac{h_{K^*}(4k\pi/n) + h_{K^*}(-4k \pi/n)}{2 \cos 4k\pi/n} -
  h_{K^*}(0) \leq \frac{\sigma}{\sqrt{4k-1}} \leq \frac{\sigma}{\sqrt{3k}}. 
\end{equation*}
Now for every 
\begin{equation}\label{ghos}
  k \leq \frac{12n }{16(1+2\sqrt{3} )^2}-1, 
\end{equation}
we have
\begin{equation*}
  \Delta_k + \frac{2\sigma}{\sqrt{k+1}} \geq \frac{2\sigma}
    {\sqrt{k+1}} \geq \frac{\sigma}{\sqrt{3n/16}} + \frac{2\sigma}
    {\sqrt{n/16}} > \Delta_{n/16} + \frac{2\sigma}{\sqrt{n/16+1}}. 
\end{equation*}
It follows therefore that any $k$ satisfying~\eqref{ghos} cannot be a
minimizer of $\Delta_k + 2\sigma (k+1)^{-1/2}$, thereby
implying~\eqref{will}. 

Let $L^*$ be defined as the Minkowski sum of $K^*$ and the closed ball
with center 0 and radius $\sigma (3n/2)^{-1/2}$. In other words,
$L^* := \left\{x + \sigma (3n/2)^{-1/2} y : x \in K \text{ and }
  ||y|| \leq 1 \right\}$. The support function $L^*$ can be
checked to equal: 
\begin{equation*}
  h_{L^*}(\theta) = h_{K^*}(\theta) + \sigma (3n/2)^{-1/2}. 
\end{equation*}
Le Cam's bound again gives
\begin{equation}\label{jp}
  r \geq \frac{1}{4} \left(h_{K^*}(0) -
    h_{L^*}(0) \right)^2 \left\{1 - ||P_{K^*} - P_{L^*}||_{TV}
  \right\}
\end{equation}
where $r$ is as defined in \eqref{rde}. By use of Pinsker's
inequality, we have 
\begin{equation*}
  ||P_{K^*} - P_{L^*}||_{TV} \leq \frac{1}{2\sigma}
  \sqrt{\sum_{i=1}^{n} \left(h_K(2i\pi/n-\pi) - h_{\breve{K}}(2i\pi/n-\pi)
    \right)^2} = \frac{1}{2\sigma} \sqrt{\frac{n\sigma^2}{3n/2}}
  \leq \frac{1}{2}. 
\end{equation*}
Therefore, from~\eqref{jp} and~\eqref{will}, we get that 
\begin{equation*}
  r \geq \frac{\sigma^2}{12n} \geq \frac{1}{16(1+2\sqrt{3})^2}
  \frac{\sigma^2}{k_*+1}. 
\end{equation*}
This completes the  proof of Theorem \ref{lobo}.  

\subsection{Proof of Theorem \ref{nt}} 
Recall the definition of $\tilde{h}^P$ in \eqref{vsp}  and the
definition of the estimator $\hat{K}$ in \eqref{khdef}. The first
thing to note is that 
\begin{equation}\label{tenu}
  h_{\hat{K}}(\theta_i) = \hat{h}_i^P \qt{for every $i = 1, \dots, n$}. 
\end{equation}
To see this, observe first that, because $\hat{h}^P = (\hat{h}_1^P,
\dots, \hat{h}_n^P)$ is a valid support vector, there exists a set
$\tilde{K}$ with $h_{\tilde{K}}(\theta_i)  = \hat{h}_i^P$ for every $i$. It
is now trivial  
(from the definition of $\hat{K}$) to see that $\tilde{K} \subseteq
\hat{K}$ which implies that $h_{\hat{K}(\theta_i)} \geq
h_{\tilde{K}}(\theta_i) =  \hat{h}_i^P$. On the other hand, the definition
of $\hat{K}$ immediately gives $h_{\hat{K}}(\theta_i) \leq
\hat{h}_i^P$. 

The observation \eqref{tenu} immediately gives
\begin{equation*}
  \E_{K^*} L_f(K^*, \hat{K}) =  \E_{K^*} \frac{1}{n} \sum_{i=1}^n
  \left(h_{K^*}(\theta_i) - \hat{h}_i^P \right)^2 
\end{equation*}
It will be convenient here to introduce the following
notation. Let $h^{vec}_{K^*}$ denote the vector $(h_{K^*}(\theta_1),
\dots, h_{K^*}(\theta_n))$. Also, for $u, v \in \R^n$, let $\ell(u,
v)$ denote the scaled Euclidean   distance defined by $\ell^2(u, v) :=
\sum_{i=1}^n (u_i - v_i)^2/n$. With this notation, we have 
\begin{equation}\label{tanu}
\E_{K^*} L_f(K^*, \hat{K}) = \E_{K^*} \ell^2(h_{K^*}^{vec}, \hat{h}^P).    
\end{equation}
Recall that $\hat{h}^P$ is the projection of $\hat{h} := (\hat{h}_1,
\dots, \hat{h}_n)$ onto $\Hs$. Because $\Hs$ is a closed convex subset
of $\R^n$, it follows that (see, for example, \citet{stark1998vector}) 
\begin{equation*}
  \ell^2(h, \hat{h}) \geq \ell^2(\hat{h}, \hat{h}^P) +
  \ell^2(h, \hat{h}^P) \qt{for every $h \in \Hs$}. 
\end{equation*}
In particular, with $h = h_{K^*}^{vec}$, we obtain
$\ell^2(h_{K^*}^{vec}, \hat{h}^P) \leq \ell^2(h_{K^*}^{vec},
\hat{h})$. Combining this with \eqref{tanu}, we obtain 
\begin{equation}\label{spac}
  \E_{K^*} L_f(K^*, \hat{K}) \leq \E_{K^*} \ell^2 (h_{K^*}^{vec},
  \hat{h}) = \frac{1}{n} \sum_{i=1}^n \E_{K^*} \left(\hat{h}_i -
    h_{K^*}(\theta_i) \right)^2. 
\end{equation}
In Theorem \ref{rbe}, we proved that  
\begin{equation*}
  \E_{K^*} \left(\hat{h}_i - h_{K^*}(\theta_i) \right)^2 \leq \frac{C
    \sigma^2}{k_*(i) + 1} \qt{for every $i = 1, \dots, n$}. 
\end{equation*}
This implies that 
\begin{equation*}
\E_{K^*} L_f(K^*, \hat{K}) \leq \frac{C
    \sigma^2}{n} \sum_{i=1}^n \frac{1}{k_*(i) + 1}. 
\end{equation*}
For inequality \eqref{nt.eq}, it is therefore enough to prove that 
\begin{equation}\label{bhagat}
  \sum_{i=1}^n \frac{1}{k_*(i) + 1} \leq C \left\{1 + \left(\frac{R
        \sqrt{n}}{\sigma} \right)^{2/5} \right\}. 
\end{equation}
Our following proof of \eqref{bhagat} is inspired by an argument due
to  \citet[Theorem 2.1]{Zhang02} in a very different context.  

Recall that $k_*(i)$ takes values in $\I := \{0\} \cup \{2^j: j \geq
0, 2^j \leq \lfloor n/16 \rfloor \}$. For $k \in \I$, let  
\begin{equation*}
  \rho(k) := \sum_{i=1}^n I\{k_*(i) = k\} ~~~~~~ \text{ and } ~~~~~
  \ell(k) := \sum_{i=1}^n I \{k_*(i) <  k  \}
\end{equation*}
Note that $\ell(0) = 0, \ell(1) = \rho(0)$ and $\rho(k) = \ell(2k) -
\ell(k)$ for $k \geq 1, k \in \I$. As a result
\begin{equation*}
  \sum_{i=1}^n \frac{1}{k_*(i) + 1} = \sum_{k \in \I}
  \frac{\rho(k)}{k+1} = \ell(1) + \sum_{k \geq 1, k \in \I}
  \frac{\ell(2k) - \ell(k)}{k+1}. 
\end{equation*}
Let $K$ denote the maximum element of $\I$. Because $\ell(2K) = n$, we
can write
\begin{equation*}
  \sum_{i=1}^n \frac{1}{k_*(i) + 1} = \frac{n}{K+1} +
  \frac{\ell(1)}{2} + \sum_{k \geq 2, k \in \I} \frac{k
    \ell(k)}{(k+1)(k+2)}. 
\end{equation*}
Using $n/(K+1) \leq C$ and loose bounds for the other terms above, we
obtain 
\begin{equation}\label{yapa}
  \sum_{i=1}^n \frac{1}{k_*(i) + 1} \leq C + \sum_{k \geq 1, k \in \I}
  \frac{3\ell(k)}{k}.  
\end{equation}
We shall show below that 
\begin{equation}\label{toso}
  \ell(k) \leq \min \left(n, \frac{A R k^{5/2}}{\sigma n} \right)
  \qt{for all $k \in \I$}
\end{equation}
for a universal positive constant $A$. Before that, let us first prove
\eqref{bhagat}  assuming \eqref{toso}. Assuming \eqref{toso}, we can
write
\begin{equation}\label{gosw}
  \sum_{k \geq 1, k \in \I} \frac{\ell(k)}{k}  =   \sum_{k \geq 1, k
    \in \I} \frac{\ell(k)}{k} I \left\{k \leq \left(\frac{\sigma n^2}{A R}
    \right)^{2/5} \right\}   +   \sum_{k \geq 1, k \in \I}
  \frac{\ell(k)}{k}  I \left\{k > \left(\frac{\sigma n^2}{A R}
    \right)^{2/5} \right\} 
\end{equation}
In the first term on the right hand side above, we use the bound
$\ell(k) \leq A R k^{5/2}/(\sigma n)$. We then get
\begin{equation*}
  \sum_{k \geq 1, k
    \in \I} \frac{\ell(k)}{k} I \left\{k \leq \left(\frac{\sigma n^2}{A R}
    \right)^{2/5} \right\}  \leq \frac{A R}{\sigma n} \sum_{k \geq 1,
    k \in \I} k^{3/2} I \left\{k \leq \left(\frac{\sigma n^2}{A R}
    \right)^{2/5} \right\}. 
\end{equation*}
Because $\I$ consists of integers of the form $2^j$, the sum in the
right hand side above is bounded from above by a constant multiple of
the last term. This gives
\begin{equation}\label{aap1}
    \sum_{k \geq 1, k
    \in \I} \frac{\ell(k)}{k} I \left\{k \leq \left(\frac{\sigma n^2}{A R}
    \right)^{2/5} \right\}  \leq \frac{CR}{\sigma n}
  \left(\frac{\sigma n^2}{A R} \right)^{3/5} =
C  \left(\frac{R\sqrt{n}}{\sigma} \right)^{2/5} 
\end{equation}
For the second term on the right hand side in \eqref{gosw}, we use the
bound $\ell(k) \leq n$ which gives
\begin{equation*}
   \sum_{k \geq 1, k
    \in \I} \frac{\ell(k)}{k} I \left\{k > \left(\frac{\sigma n^2}{A R}
    \right)^{2/5} \right\}  \leq n \sum_{k \geq 1, k
    \in \I} k^{-1} I \left\{k > \left(\frac{\sigma n^2}{A R}
    \right)^{2/5} \right\} 
\end{equation*}
Again, because $\I$ consists of integers of the form $2^j$, the sum in
the right hand side above is bounded from above by a constant multiple
of the first term. This gives
\begin{equation}\label{aap2}
  \sum_{k \geq 1, k
    \in \I} \frac{\ell(k)}{k} I \left\{k > \left(\frac{\sigma n^2}{A R}
    \right)^{2/5} \right\}  \leq C n \left(\frac{\sigma n^2}{A R}
  \right)^{-2/5} = C  \left(\frac{R\sqrt{n}}{\sigma} \right)^{2/5} . 
\end{equation}
Inequalities \eqref{aap1} and \eqref{aap2} in conjunction with
\eqref{yapa} proves \eqref{bhagat} which would complete the proof of
\eqref{nt.eq}. 

We only need to prove \eqref{toso}. For this, observe first that when
$k_*(i) < k$, Corollary \ref{arbe} gives that 
\begin{equation}\label{coro}
  \Delta_k(\theta_i) \geq \frac{(\sqrt{6} - 2) \sigma}{\sqrt{k + 1}}. 
\end{equation}
This is because if \eqref{coro} is violated, then Corollary \ref{arbe}
gives $k \leq \breve{k}(i) \leq k_*(i)$. Consequently, we have 
\begin{equation*}
  I\{k_*(i) < k\} \leq \frac{\Delta_k(\theta_i) \sqrt{k+1}}{(\sqrt{6}
    - 2) \sigma}
\end{equation*}
and 
\begin{equation}\label{elb}
  \ell(k) \leq \frac{\sqrt{k+1}}{(\sqrt{6} - 2) \sigma} \sum_{i=1}^n
  \Delta_k(\theta_i) \qt{for every $k \in \I$}. 
\end{equation}
Now using the expression \eqref{alex} for $\Delta_k(\theta_i)$, it is
easy to see that 
\begin{equation}\label{kart}
  \sum_{i=1}^n \Delta_k(\theta_i) = \frac{1}{k+1} \sum_{j=0}^k \delta_j
\end{equation}
where $\delta_j$ is given by 
\begin{equation*}
\delta_j = \sum_{i=1}^n \left(\frac{h_{K^*}(\theta_i+ 4j\pi/n) + h_{K^*}(\theta_i
    -4j \pi/n)}{2} - \frac{\cos(4j\pi/n)}{\cos(2j\pi/n)} 
    \frac{h_{K^*}(\theta_i+2j\pi/n) + h_{K^*}(\theta_i -2j\pi/n)}{2}
  \right).  
\end{equation*}
We will now prove an upper bound for $\delta_j$ under the assumption
that $K^*$ is contained in a ball of radius $R \geq 0$. We may assume
without loss of generality that this ball is centered at the origin
because the expression for $\delta_j$ above remains unchanged if
$h_{K^*}(\theta)$ is replaced by $h_{K^*}(\theta) - a_1 \cos \theta -
a_2 \sin \theta$ for any $(a_1, a_2) \in \R^2$. Because $\theta_i =
2\pi i/n - \pi$, we can rewrite $\delta_j$ as 
\begin{equation*}
\delta_j = \sum_{i=1}^n \left(\frac{h_{K^*}(\theta_{i+2j}) +
    h_{K^*}(\theta_{i-2j})}{2} - \frac{\cos(4j\pi/n)}{\cos(2j\pi/n)}  
    \frac{h_{K^*}(\theta_{i+j}) + h_{K^*}(\theta_{i-j})}{2}
  \right).  
\end{equation*}
Because $\theta \mapsto h_{K^*}(\theta)$ is a periodic function of
period $2 \pi$, the above expression only depends on
$h_{K^*}(\theta_1)$, ..., $h_{K^*}(\theta_n)$. In fact, it is easy to
see that 
\begin{equation*}
  \delta_j = \left(1 - \frac{\cos (4 j \pi/n)}{\cos (2 j \pi/n)}
  \right) \sum_{i=1}^n h_{K^*}(\theta_i). 
\end{equation*}
Now because $K^*$ is contained in the ball of radius $R$ centered at
the origin, it follows that $|h_{K^*}(\theta_i)| \leq R$ for each $i$
which gives
\begin{equation*}
  \delta_j \leq n R \left(1 - \frac{\cos (4 j \pi/n)}{\cos (2 j
      \pi/n)} \right) \leq n R \left(1 - \frac{\cos (4 k \pi/n)}{\cos (2 k
      \pi/n)} \right) = \frac{n R(1 + 2 \cos 2\pi k/n)}{\cos 2 \pi k/n}
  (1 - \cos 2\pi k/n)
\end{equation*}
for all $0 \leq j \leq k$. Because $k \leq n/16$ for all $k \in \I$,
it follows that  
\begin{equation*}
  \delta_j \leq 8 n R \sin^2 (\pi k/n) \leq \frac{8 R \pi^2 k^2}{n}
  \qt{for all $0 \leq j \leq k$}. 
\end{equation*}
The identity \eqref{kart} therefore gives $\sum_{i=1}^n
\Delta_k(\theta_i) \leq 8 R \pi^2 k^2/n$ for all $k \in
\I$. Consequently, from \eqref{elb} and the trivial fact that $\ell(k)
\leq n$, we obtain
\begin{equation*}
  \ell(k) \leq \min \left(n,  \frac{8 \pi^2}{(\sqrt{6} - 2)}
  \frac{R k^2\sqrt{k+1}}{\sigma n} \right) \qt{for all $k \in \I$}. 
\end{equation*}
Note that $\ell(0) = 0$ so that the above inequality only gives
something useful for $k \geq 1$. Using $k + 1 \leq 2 k$ for $k \geq 1$
and denoting the resulting constant by $C$, we obtain
\eqref{toso}. This completes the proof of Theorem \ref{nt}.

\subsection{Proof of Theorem \ref{theorem::General}}
The following lemma will be crucially used in our proof of Theorem
\ref{theorem::General}. For every compact, convex set $P$ and $i = 1,
\dots, n$, let $k_*^P(i)$  denote the quantity $k_*$ with $K^*$ replaced by 
$P$. More precisely, 
\begin{equation*}
  k_*^P(i) := \argmin_{k \in \I} \left(\Delta_k^P(\theta_i) + \frac{2
      \sigma}{\sqrt{k+1}} \right) 
\end{equation*}
where $\Delta_k^P(\theta_i)$ is given by 
\begin{equation*}
  \frac{1}{k+1} \sum_{j=0}^k \left( \frac{h_{P}(\theta_i+ 4j\pi/n) +
      h_{P}(\theta_i -4j \pi/n)}{2} -
    \frac{\cos(4j\pi/n)}{\cos(2j\pi/n)}
    \frac{h_{P}(\theta_i+2j\pi/n) + h_{P}(\theta_i
      -2j\pi/n)}{2}\right). 
\end{equation*}
The next lemma states that for every $i = 1, \dots, n$, the risk
$\E_{K^*} (\hat{h}_i - h_{K^*}(\theta_i))^2$ can be bounded from above
by a combination of $k_*^P(i)$ and how well $K^*$ can be approximated
by $P$. This result holds for every $P$. The approximation of $K^*$ by
$P$ is measured in terms of the Hausdorff distance (defined in
\eqref{hausdef}). 

\begin{lemma}[Approximation]\label{appx}
  There exists a universal positive constant $C$ such that for every
  $i = 1, \dots, n$ and every compact, convex set $P$, we have  
  \begin{equation}\label{appx.eq}
     \E_{K^*} \left(\hat{h}_i - h_{K^*}(\theta_i) \right)^2 \leq C
     \left(\frac{\sigma^2}{k_*^P(i) + 1} + \ell_H^2(K^*, P) \right). 
  \end{equation}
\end{lemma}

\begin{proof}[Proof of Lemma \ref{appx}]
  Fix $i \in \{1, \dots, n\}$ and a compact, convex set $P$. For
  notational convenience, we write $\Delta_k, \Delta_k^P, k_*$ and
  $k_*^P$ for $\Delta_k(\theta_i), \Delta_k^P(\theta_i),
  k_*(\theta_i)$ and $k_*^P(\theta_i)$ respectively. 

  We assume that the following condition holds:
  \begin{equation}\label{conc}
    k_*^P + 1 \geq \frac{24(\sqrt{2} - 1)}{\sqrt{6} - 2} (k_* + 1). 
  \end{equation}
  If this condition does not hold, we have
  \begin{equation*}
    \frac{1}{k_* + 1} < \frac{24(\sqrt{2} - 1)}{\sqrt{6} - 2}
    \frac{1}{k_*^P + 1}
  \end{equation*}
  and then \eqref{appx} immediately follows from Theorem \ref{rbe}. 

  Note that \eqref{conc} implies, in particular, that $k_*^P >
  k_*$. Inequality \eqref{an.delkstar.eq} in Lemma \ref{delkstar}
  applied to $k = k_*^P$ implies therefore that
  \begin{equation*}
    \Delta_{k_*^P} \geq \frac{(\sqrt{6} - 2) \sqrt{k_*^P + 1}
      \sigma}{2(k_* + 1)}. 
  \end{equation*}
  Also inequality \eqref{delkstar.eq} applied to the set $P$ instead
  of $K^*$ gives 
  \begin{equation*}
    \Delta^P_{k_*^P} \leq \frac{6(\sqrt{2} - 1) \sigma}{\sqrt{k_*^P +
        1}}. 
  \end{equation*}
  Combining the above pair of inequalities, we obtain 
  \begin{equation*}
    \Delta_{k_*^P} - \Delta_{k_*^P}^P \geq \frac{(\sqrt{6} - 2) \sqrt{k_*^P + 1}
      \sigma}{2(k_* + 1)} - \frac{6(\sqrt{2} - 1) \sigma}{\sqrt{k_*^P +
        1}}. 
  \end{equation*}
  The right hand above is non-decreasing in $k_*^P + 1$ and so we can
  replace $k_*^P + 1$ by the lower bound in \eqref{conc} to obtain,
  after some simplication, 
  \begin{equation}\label{mmd}
    \Delta_{k_*^P} - \Delta_{k_*^P}^P \geq \frac{\sigma}{4 \sqrt{k_* +
      1}} \sqrt{24(\sqrt{2} - 1)(\sqrt{6} - 2)}. 
  \end{equation}
  The key now is to observe that 
  \begin{equation}\label{obsb}
    |\Delta_k - \Delta_k^P| \leq 2 \ell_H(K^*, P) \qt{for all $k$}. 
  \end{equation}
  This follows from the definition \eqref{hausdef} of the Hausdorff
  distance which gives 
  \begin{equation*}
    \left| \Delta_{k} - \Delta_k^{P} \right| \leq \ell_H(K^*, P)
    \left(1 + \frac{1}{k+1} \sum_{j=0}^k \frac{\cos (4j\pi/n)}{\cos
        (2j \pi/n)} \right)
  \end{equation*}
  and this clearly implies \eqref{obsb} because $\cos (4 j \pi/n)/\cos
  (2j \pi/n) \leq 1$ for all $0 \leq j \leq k$. 

   From \eqref{obsb} and \eqref{mmd}, we deduce that 
   \begin{equation*}
     \ell_H(K^*, P) \geq \frac{c \sigma}{\sqrt{k_* + 1}} 
   \end{equation*}
  for a universal positive constant $c$. This, together with
  inequality \eqref{rbe.eq}, clearly implies \eqref{appx.eq} which
  completes the proof. 
\end{proof}

We are now ready to prove Theorem \ref{theorem::General}. 
\begin{proof}[Proof of Theorem \ref{theorem::General}]
We use inequality \eqref{spac} from the proof of Theorem
\ref{nt}. This inequality, along with \eqref{appx.eq} for $i = 1,
\dots, n$, gives 
\begin{equation*}
  \E_{K^*} L_f \left(K^*, \hat{K} \right) = \frac{1}{n}
  \sum_{i=1}^n \E_{K^*} \left(\hat{h}_i - h_{K^*}(\theta_i) \right)^2
  \leq C \left(\frac{\sigma^2}{n} \sum_{i=1}^n \frac{1}{k_*^P(i) + 1}
    + \ell_H^2(K^*, P) \right)
\end{equation*}
for every compact, convex set $P$. By restricting $P$ to be in the
class of polytopes, we get
\begin{equation*}
  \E_{K^*} L_f \left(K^*, \hat{K} \right) \leq C \inf_{P \in \Ps} 
 \left(\frac{\sigma^2}{n} \sum_{i=1}^n \frac{1}{k_*^P(i) + 1} +
    \ell_H^2(K^*, P) \right). 
\end{equation*}
For the proof of \eqref{GloGen}, it is therefore enough to show that 
\begin{equation}\label{jk}
  \sum_{i=1}^n \frac{1}{k^P_*(i) + 1} \leq C v_P \log (en/v_P) \qt{for
    every $P \in \Ps$}
\end{equation}
where $v_P$ denotes the number of extreme points of $P$ and $C$ is a
universal positive constant.  Fix a polytope $P$ with $v_P = k$. Let
the extreme points of $P$ be $z_1, \dots, z_k$. Let $S_1, \dots, S_k$
denote a partition of $\{\theta_1, \dots, \theta_n\}$ into $k$
nonempty sets such that for each $j = 1, \dots, m$, we have  
\begin{equation*}
  h_P(\theta_i) = z_j(1) \cos \theta_i + z_j(2) \sin \theta_i 
  \qt{for all $\theta_i \in S_j$}
\end{equation*}
where $z_j = (z_j(1), z_j(2))$. For \eqref{jk}, it is enough to prove
that 
\begin{equation}\label{jkm}
  \sum_{i: \theta_i \in S_j} \frac{1}{k^P_*(i) + 1} \leq C \log (en_j)
  \qt{for every $j = 1, \dots, k$}
\end{equation}
where $n_j$ is the cardinality of $S_j$. This is because we can write
\begin{equation*}
  \sum_{i=1}^n \frac{1}{k_*^P(i) + 1}  = \sum_{j=1}^k \sum_{i:
    \theta_i \in S_j} \frac{1}{k_*^P(i) + 1} \leq C \sum_{j=1}^k \log
  (en_j) \leq C k \log(en/k)
\end{equation*}
where we used the concavity of $x \mapsto \log(ex)$. We prove
\eqref{jkm} below. Fix $1 \leq j \leq k$. The inequality is obvious if
$S_j$ is a singleton because $k_*^P(i) \geq 
0$. So suppose that $n_j = m \geq 2$. Without loss of
generality assume that $S_j = \{\theta_{u+1}, \dots, \theta_{u+m}\}$
where $0 \leq u \leq n-m$. The definition of $S_j$ implies that
\begin{equation*}
  h_{P}(\theta) = z_j(1) \cos \theta + z_j(2) \sin \theta \qt{for
    all $\theta \in [\theta_{u+1}, \theta_{u+m}]$}. 
\end{equation*}
We can therefore apply inequality \eqref{vidu.eq} to claim the
existence of a positive constant $c$ such that 
\begin{equation*}
  k_*^P(i) \geq c ~ n \min \left(\theta_i - \theta_{u+1}, \theta_{u+m}
  - \theta_i \right) \qt{for all $u+1 \leq i \leq u+m$}. 
\end{equation*}
The minimum with $\pi$ in \eqref{vidu.eq} is redundant here because
$\theta_{u+m} - \theta_{u+1} < 2 \pi$. Because $\theta_i = 2\pi i/n -
\pi$, we get 
\begin{equation*}
  k_*^P(i) \geq 2 \pi c \min \left(i-u-1, u+m-i \right) \qt{for all $u +
    1 \leq i \leq u+m$}.  
\end{equation*}
Therefore, there exists a universal constant $C$ such that 
\begin{equation*}
  \sum_{i: \theta_i \in S_j} \frac{1}{k^P_*(i) + 1} \leq C
  \sum_{i=1}^m  \frac{1}{1 + \min(i-1, m-i)} \leq C
  \sum_{i=1}^{m} \frac{1}{i} \leq C \log (em). 
\end{equation*}
This proves \eqref{jkm} thereby completing the proof of Theorem
\ref{theorem::General}. 
\end{proof}

\subsection{Proof of Theorem \ref{nti}}
Recall the definition \eqref{khpdef} of the estimator $\hat{K}'$ and
that of the interpolating function \eqref{hind}. Following an argument
similar to that used at the beginning of the proof of Theorem
\ref{nt}, we observe that
\begin{equation}\label{fobs}
  \E_{K^*} L(K^*, \hat{K}') \leq \int_{-\pi}^{\pi} \E_{K^*}
  \left(h_{K^*}(\theta) - \hat{h}'(\theta) \right)^2 d\theta =
  \sum_{i=1}^{n} \int_{\theta_i}^{\theta_{i+1}}  \E_{K^*}
  \left(h_{K^*}(\theta) - \hat{h}'(\theta) \right)^2 d\theta 
\end{equation}
Now fix $1 \leq i \leq n$, $\theta_i \leq \theta \leq
\theta_{i+1}$ and let $u(\theta) := \E_{K^*} \left(h_{K^*}(\theta) -
  \hat{h}'(\theta) \right)^2$. Using the expression \eqref{hind} for
$\hat{h}'(\theta)$, we get that
\begin{equation*}
  u(\theta) = \E_{K^*} \left(h_{K^*}(\theta) - \frac{\sin(\theta_{i+1}
      - \theta)}{\sin (\theta_{i+1} - \theta_i)} \hat{h}_i -
    \frac{\sin(\theta - \theta_i)}{\sin(\theta_{i+1} - \theta_i)}
    \hat{h}_{i+1} \right)^2 . 
\end{equation*}
We now write $\hat{h}_i = \hat{h}_i - h_{K^*}(\theta_i) +
h_{K^*}(\theta_i)$ and a similar expression for $\hat{h}_{i+1}$. The
elementary inequality $(a + b + c)^2 \leq 3(a^2+b^2+c^2)$ along with
$\max \left( \sin(\theta - \theta_i), \sin(\theta_{i+1} -
  \theta)\right) \leq \sin(\theta_{i+1} - \theta_i)$ then imply that 
\begin{equation*}
  u(\theta) \leq 3 \E_{K^*} \left(\hat{h}_i - h_{K^*}(\theta_i)
  \right)^2 + 3 \E_{K^*} \left(\hat{h}_{i+1} - h_{K^*}(\theta_{i+1})
  \right)^2 + 3 b^2(\theta)
\end{equation*}
where 
\begin{equation*}
  b(\theta) := h_{K^*}(\theta) - \frac{\sin(\theta_{i+1}
      - \theta)}{\sin (\theta_{i+1} - \theta_i)} h_{K^*}(\theta_i) -
    \frac{\sin(\theta - \theta_i)}{\sin(\theta_{i+1} - \theta_i)}
    h_{K^*}(\theta_{i+1}) 
\end{equation*}
Therefore from \eqref{fobs} (remember that $|\theta_{i+1} - \theta_i|
= 2 \pi/n$), we deduce 
\begin{equation*}
  \E_{K^*} L(K^*, \hat{K}') \leq \frac{12 \pi}{n} \sum_{i=1}^n
  \E_{K^*} \left(\hat{h}_i - h_{K^*}(\theta_i) \right)^2 + 3
  \int_{-\pi}^{\pi} b^2(\theta) d\theta. 
\end{equation*}
Now to bound $\sum_{i=1}^n \E_{K^*} \left(\hat{h}_i -
  h_{K^*}(\theta_i) \right)^2$, we can simply use the arguments from
the proofs of Theorems \ref{nt} and \ref{theorem::General}. Therefore,
to complete the proof of Theorem \ref{nti}, we only need to show that
\begin{equation}\label{bth}
  |b(\theta)| \leq \frac{C R}{n} \qt{for every $\theta \in (-\pi, \pi]$}
\end{equation}
for some universal constant $C$.  For this, we use the hypothesis that
$K^*$ is contained in a ball of radius $R$. Suppose that the center of
the ball is $(x_1, x_2)$. Define $K' := K^* - \{(x_1, x_2)\} :=
\{(y_1, y_2) - (x_1, x_2) : (y_1, y_2) \in K^* \}$ and note that
$h_{K'}(\theta) = h_{K^*}(\theta) - x_1 \cos \theta - x_2 \sin
\theta$. It is then easy to see that $b(\theta)$ is the same for both
$K^*$ and $K'$. It is therefore enough to prove \eqref{bth} assuming
that $(x_1, x_2) = (0, 0)$.  In this case, it is straightforward to
see that $|h_{K^*}(\theta)| \leq R$ for all $\theta$ and also that
$h_{K^*}$ is Lipschitz with constant $R$. Now, because $\max \left(
  \sin(\theta - \theta_i), \sin(\theta_{i+1} - \theta)\right) \leq
\sin(\theta_{i+1} - \theta_i)$, it can be checked that 
\begin{equation*}
  |b(\theta)| \leq |h_{K^*}(\theta)| \left|1 - \frac{\sin(\theta_{i+1}
      - \theta)}{\sin(\theta_{i+1} - \theta_i)} - \frac{\sin(\theta -
      \theta_i)}{\sin(\theta_{i+1} - \theta_i)}  \right| +
  |h_{K^*}(\theta_i) - h_{K^*}(\theta)| + |h_{K^*}(\theta_{i+1}) -
  h_{K^*}(\theta)|. 
\end{equation*}
Because $h_{K^*}$ is $R$-Lipschitz and bounded by $R$, it is clear
that we only need to show 
\begin{equation*}
  \left|1 - \frac{\sin(\theta_{i+1}
      - \theta)}{\sin(\theta_{i+1} - \theta_i)} - \frac{\sin(\theta -
      \theta_i)}{\sin(\theta_{i+1} - \theta_i)}  \right| \leq \frac{C}{n}
\end{equation*}
in order to prove \eqref{bth}.  For this, write $\alpha = \theta_{i+1}
- \theta$ and $\beta = \theta - \theta_i$ so that the above expression
becomes
\begin{equation*}
  \left|1 - \frac{\sin \alpha + \sin \beta}{\sin(\alpha + \beta)}
  \right| \leq |1 - \cos \alpha| + |1 - \cos \beta| \leq
  \frac{\alpha^2 + \beta^2}{2} \leq \frac{C}{n^2} \leq \frac{C}{n}. 
\end{equation*}
This completes the proof of Theorem \ref{nti}.

\subsection{Proofs of Corollaries in Section \ref{cors}} 

The proofs of the corollaries stated in Section \ref{cors} are given here. For these proofs, we need some simple properties of the $\Delta_k(\theta_i)$ which are stated and proved in Appendix \ref{apap}.  

We start with the proof of Corollary \ref{arbe}. 
\begin{proof}[Proof of Corollary \ref{arbe}]
Fix $1 \leq i \leq n$. We will prove that $\breve{k}(i) \leq k_*(i) \leq
\tilde{k}(i)$. Inequality \eqref{arbe.eq} would then follow from
Theorem \ref{rbe}. For simplicity, we write $\Delta_k$ for
$\Delta_k(\theta_i)$,  $f_k$ for $f_k(\theta_i)$, $g_k$ for
$g_k(\theta_i)$, $k_*$ for $k_*(i)$, $\breve{k}$ for $\breve{k}(i)$
and $\tilde{k}$ for $\tilde{k}(i)$. 

Inequality \eqref{an.delkstar.eq} in Lemma \ref{delkstar} gives 
\begin{equation*}
  \Delta_k \geq \frac{\sigma (\sqrt{6} - 2)}{\sqrt{k + 1}} \qt{for all
  $k > k_*, k \in \I$}. 
\end{equation*}
Thus any $k \in \I$ for which $f_k \leq \Delta_k < \sigma(\sqrt{6} -
2)/\sqrt{k+1}$ has to satisfy $k \leq k_*$. This proves $\breve{k}
\leq k_*$. 

For $k_* \leq \tilde{k}$, we first inequality \eqref{delkstar.eq} in
Lemma \ref{delkstar} to obtain $\Delta_{k_*} \geq 6(\sqrt{2} - 1)
\sigma/\sqrt{k_*+1}$. Also Lemma \ref{keydrop} states that $k \mapsto
\Delta_k$ is non-decreasing for $k \in \I$. We therefore have 
\begin{equation*}
  g_k \leq \Delta_k \leq \Delta_{k_*} \leq \frac{6(\sqrt{2} - 1)
    \sigma}{\sqrt{k_*+1}} \leq \frac{6(\sqrt{2} - 1)
    \sigma}{\sqrt{k+1}} \qt{for all $k \leq k_*, k \in \I$}. 
\end{equation*}
Therefore any $k \in \I$ for which $g_k > 6(\sqrt{2} - 1)
\sigma/\sqrt{k+1}$ has to be larger than $k_*$. This proves $\tilde{k}
\geq k_*$. The proof is complete.
\end{proof}

We next give the proof of Corollary \ref{theorem::RadiusRBall}. 
\begin{proof}[Proof of Corollary \ref{theorem::RadiusRBall}]
We only need to prove \eqref{f1.eq}. Inequality \eqref{lp1.eq} would
then follow from Theorem \ref{rbe}. Fix $i \in \{1, \dots, n\}$ and
suppose that $K^*$ is contained in a ball of radius $R$ centered at
$(x_1, x_2)$.  We shall prove below that
$\Delta_k(\theta_i) \leq 6 \pi R k/n$ for every $k \in \I$ and
\eqref{f1.eq} would then follow from Corollary \ref{arbe}. Without
loss of generality, assume that $\theta_i = 0$.  

As in the proof of Theorem \ref{nti}, we may assume that $K^*$ is
contained in the ball of radius $R$ centered at the origin. This
implies that $|h_{K^*}(\theta)| \leq R$ for all $\theta$ and also that
$h_{K^*}$ is Lipschitz with constant $R$. Note then that for every $k
\in \I$ and $0 \leq j \leq k$, the quantity  
\begin{equation*}
 Q := \frac{h_{K^*}(4j\pi/n) + h_{K^*}(-4j\pi/n)}{2} - \frac{\cos(4 j \pi/n)}{\cos(2 j
    \pi/n)} \frac{h_{K^*}(2j\pi/n) + h_{K^*}(-2j\pi/n)}{2}
\end{equation*}
can be bounded as
\begin{align*}
  |Q| &= \left\vert \frac{h_{K^*}(4j\pi/n)- h_{K^*}(2j\pi/n)  + h_{K^*}(-4j \pi/n)-
      h_{K^*}(-2j\pi/n)}{2} \right.
  \\
\notag &-
\left. \left(\frac{\cos(4j\pi/n)-\cos(2j\pi/n)}{\cos(2j\pi/n)}\right)\frac{h_{K^*}(2j\pi/n)
    + h_{K^*}(-2j\pi/n)}{2} \right\vert \leq \frac{6Rj \pi}{n}.  
\end{align*}
Here we used also the fact that $\cos(\cdot)$ is Lipschitz and $\cos
(2 j\pi/n) \geq 1/2$. The inequality $\Delta_k(0) \leq 6\pi Rk/n$ then
immediately follows. The proof is complete. 
\end{proof}

We conclude this section with a proof of Corollary
\ref{theorem::individual}. 
\begin{proof}[Proof of Corollary \ref{theorem::individual}]
By Theorem \ref{rbe}, inequality \eqref{indi.eq}  is a direct
consequence of \eqref{vidu.eq}. We therefore only need to prove
\eqref{vidu.eq}. Fix $k \in \I$ with
\begin{equation}\label{kist}
  k \leq \frac{n}{4 \pi} \min(\theta_i - \phi_1(i), \phi_2(i) -
  \theta_i). 
\end{equation}
It is then clear that $\theta_i \pm 4 j \pi/n \in [\phi_1(i),
\phi_2(i)]$ for every $0 \leq j \leq k$. From \eqref{anny}, it follows
that  
\begin{equation*}
  h_{K^*}(\theta) = x_1 \cos \theta + x_2 \sin \theta \qt{for all
    $\theta = \theta_i \pm \frac{4 j \pi}{n}, 0 \leq j \leq k$}. 
\end{equation*}
We now argue that $\Delta_k(\theta_i) = 0$. To see this, note first
that $\Delta_k(\theta_i) = U_k(\theta_i) - L_k(\theta_i)$ has the
following alternative expression \eqref{alex}. Plugging in
$h_{K^*}(\theta) = x_1 \cos 
\theta + x_2 \sin \theta$ in \eqref{alex}, one can see by direct
computation that $\Delta_k(\theta_i) = 0$ for every $k \in \I$
satisfying \eqref{kist}. The definition \eqref{kst} of $k_*(i)$ now
immediately implies that  
\begin{equation*}
  k_*(i) \geq \min \left(\frac{n}{4 \pi} \min(\theta_i - \phi_1(i), \phi_2(i) -
  \theta_i), c n \right)
\end{equation*}
for a small enough universal constant $c$. This proves \eqref{vidu.eq}
thereby completing the proof. 
\end{proof}

\bibliographystyle{chicago}
\bibliography{AG}

\def\noopsort#1{}
\begin{thebibliography}{}

\bibitem[\protect\citeauthoryear{Alexandrov}{Alexandrov}{1939}]{Alexandrov39}
Alexandrov, A.~D. (1939).
\newblock Almost everywhere existence of the second differential of a convex
  function and some properties of convex surfaces connected with it.
\newblock {\em Leningrad State Univ. Annals [Uchenye Zapiski] Math.
  Ser.\/}~{\em 6}, 3--35.

\bibitem[\protect\citeauthoryear{Baraud and Birg{\'e}}{Baraud and
  Birg{\'e}}{2015}]{baraud2015rates}
Baraud, Y. and L.~Birg{\'e} (2015).
\newblock Rates of convergence of rho-estimators for sets of densities
  satisfying shape constraints.
\newblock {\em arXiv preprint arXiv:1503.04427\/}.

\bibitem[\protect\citeauthoryear{Brunel}{Brunel}{2014}]{brunel2014non}
Brunel, V.-E. (2014).
\newblock {\em Non-parametric estimation of convex bodies and convex
  polytopes}.
\newblock Ph.\ D. thesis, Universit{\'e} Pierre et Marie Curie-Paris VI;
  University of Haifa.

\bibitem[\protect\citeauthoryear{Brunk}{Brunk}{1970}]{Brunk70}
Brunk, H.~D. (1970).
\newblock Estimation of isotonic regression.
\newblock In {\em Nonparametric {T}echniques in {S}tatistical {I}nference
  ({P}roc. {S}ympos., {I}ndiana {U}niv., {B}loomington, {I}nd., 1969)}, pp.\
  177--197. London: Cambridge Univ. Press.

\bibitem[\protect\citeauthoryear{Cai and Low}{Cai and Low}{2015}]{CaiLowFwork}
Cai, T.~T. and M.~G. Low (2015).
\newblock A framework for estimation of convex functions.
\newblock {\em Statistica Sinica\/}~{\em 25}, 423--456.

\bibitem[\protect\citeauthoryear{Cai, Low, and Xia}{Cai
  et~al.}{2013}]{cai2013adaptive}
Cai, T.~T., M.~G. Low, and Y.~Xia (2013).
\newblock Adaptive confidence intervals for regression functions under shape
  constraints.
\newblock {\em Annals of Statistics\/}~{\em 41}, 722--750.

\bibitem[\protect\citeauthoryear{Carolan and Dykstra}{Carolan and
  Dykstra}{1999}]{CD99}
Carolan, C. and R.~Dykstra (1999).
\newblock Asymptotic behavior of the {G}renander estimator at density flat
  regions.
\newblock {\em Canad. J. Statist.\/}~{\em 27\/}(3), 557--566.

\bibitem[\protect\citeauthoryear{Cator}{Cator}{2011}]{Cator2011}
Cator, E. (2011).
\newblock Adaptivity and optimality of the monotone least-squares estimator.
\newblock {\em Bernoulli\/}~{\em 17}, 714--735.

\bibitem[\protect\citeauthoryear{Chatterjee, Guntuboyina, and Sen}{Chatterjee
  et~al.}{2014}]{GuntuAnnIso}
Chatterjee, S., A.~Guntuboyina, and B.~Sen (2014).
\newblock On risk bounds in isotonic and other shape restricted regression
  problems.
\newblock {\em Annals of Statistics\/}.
\newblock to appear.

\bibitem[\protect\citeauthoryear{Fisher, Hall, Turlach, and Watson}{Fisher
  et~al.}{1997}]{FisherHallTurlachWatson}
Fisher, N.~I., P.~Hall, B.~A. Turlach, and G.~S. Watson (1997).
\newblock On the estimation of a convex set from noisy data on its support
  function.
\newblock {\em Journal of the American Statistical Association\/}~{\em 92},
  84--91.

\bibitem[\protect\citeauthoryear{Gardner}{Gardner}{2006}]{GardnerBook}
Gardner, R.~J. (2006).
\newblock {\em Geometric Tomography\/} (second ed.).
\newblock Cambridge University Press.

\bibitem[\protect\citeauthoryear{Gardner and Kiderlen}{Gardner and
  Kiderlen}{2009}]{GardnerKiderlen2009}
Gardner, R.~J. and M.~Kiderlen (2009).
\newblock A new algorithm for 3{D} reconstruction from support functions.
\newblock {\em IEEE Transactions on Pattern Analysis and Machine
  Intelligence\/}~{\em 31}, 556--562.

\bibitem[\protect\citeauthoryear{Gardner, Kiderlen, and Milanfar}{Gardner
  et~al.}{2006}]{GKM06}
Gardner, R.~J., M.~Kiderlen, and P.~Milanfar (2006).
\newblock Convergence of algorithms for reconstructing convex bodies and
  directional measures.
\newblock {\em Annals of Statistics\/}~{\em 34}, 1331--1374.

\bibitem[\protect\citeauthoryear{Gregor and Rannou}{Gregor and
  Rannou}{2002}]{GregorRannou}
Gregor, J. and F.~R. Rannou (2002).
\newblock Three-dimensional support function estimation and application for
  projection magnetic resonance imaging.
\newblock {\em International Journal of Imaging Systems and Technology\/}~{\em
  12}, 43--50.

\bibitem[\protect\citeauthoryear{Groeneboom}{Groeneboom}{1983}]{G83}
Groeneboom, P. (1983).
\newblock The concave majorant of {B}rownian motion.
\newblock {\em Ann. Probab.\/}~{\em 11\/}(4), 1016--1027.

\bibitem[\protect\citeauthoryear{Groeneboom}{Groeneboom}{1985}]{G85}
Groeneboom, P. (1985).
\newblock Estimating a monotone density.
\newblock In {\em Proceedings of the {B}erkeley conference in honor of {J}erzy
  {N}eyman and {J}ack {K}iefer, {V}ol.\ {II} ({B}erkeley, {C}alif., 1983)},
  Wadsworth Statist./Probab. Ser., Belmont, CA, pp.\  539--555. Wadsworth.

\bibitem[\protect\citeauthoryear{Groeneboom and Jongbloed}{Groeneboom and
  Jongbloed}{2014}]{groeneboom2014nonparametric}
Groeneboom, P. and G.~Jongbloed (2014).
\newblock {\em Nonparametric Estimation under Shape Constraints: Estimators,
  Algorithms and Asymptotics}, Volume~38.
\newblock Cambridge University Press.

\bibitem[\protect\citeauthoryear{Groeneboom, Jongbloed, and Wellner}{Groeneboom
  et~al.}{2001a}]{GroeneboomJongbloedWellner2001a}
Groeneboom, P., G.~Jongbloed, and J.~A. Wellner (2001a).
\newblock A canonical process for estimation of convex functions: The
  "invelope" of integrated brownian motion $+t^4$.
\newblock {\em Annals of Statistics\/}~{\em 29}, 1620--1652.

\bibitem[\protect\citeauthoryear{Groeneboom, Jongbloed, and Wellner}{Groeneboom
  et~al.}{2001b}]{GroeneboomJongbloedWellner2001b}
Groeneboom, P., G.~Jongbloed, and J.~A. Wellner (2001b).
\newblock Estimation of convex functions: characterizations and asymptotic
  theory.
\newblock {\em Annals of Statistics\/}~{\em 29}, 1653--1698.

\bibitem[\protect\citeauthoryear{Guntuboyina}{Guntuboyina}{2011}]{G11}
Guntuboyina, A. (2011).
\newblock Optimal rates of convergence for the estimation of reconstruction of
  convex bodies from noisy support function measurements.
\newblock {\em Annals of Statistics\/}.
\newblock to appear.

\bibitem[\protect\citeauthoryear{Guntuboyina and Sen}{Guntuboyina and
  Sen}{2013}]{GSvex}
Guntuboyina, A. and B.~Sen (2013).
\newblock Global risk bounds and adaptation in univariate convex regression.
\newblock {\em Probab. Theory Related Fields\/}.
\newblock \textit{To appear}, available at http://arxiv.org/abs/1305.1648.

\bibitem[\protect\citeauthoryear{Hanson and Pledger}{Hanson and
  Pledger}{1976}]{HanPled76}
Hanson, D.~L. and G.~Pledger (1976).
\newblock Consistency in concave regression.
\newblock {\em Ann. Statist.\/}~{\em 4\/}(6), 1038--1050.

\bibitem[\protect\citeauthoryear{Jankowski}{Jankowski}{2014}]{H14}
Jankowski, H. (2014).
\newblock Convergence of linear functionals of the {G}renander estimator under
  misspecification.
\newblock {\em Ann. Statist.\/}~{\em 42\/}(2), 625--653.

\bibitem[\protect\citeauthoryear{Le~Cam}{Le~Cam}{1986}]{LeCam:86book}
Le~Cam, L. (1986).
\newblock {\em Asymptotic Methods in Statistical Decision Theory}.
\newblock New York: Springer-Verlag.

\bibitem[\protect\citeauthoryear{Lele, Kulkarni, and Willsky}{Lele
  et~al.}{1992}]{LeleKulkarniWillsky}
Lele, A.~S., S.~R. Kulkarni, and A.~S. Willsky (1992).
\newblock Convex-polygon estimation from support-line measurements and
  applications to target reconstruction from laser-radar data.
\newblock {\em Journal of the Optical Society of America, Series A\/}~{\em 9},
  1693--1714.

\bibitem[\protect\citeauthoryear{Mammen}{Mammen}{1991}]{Mammen91}
Mammen, E. (1991).
\newblock Nonparametric regression under qualitative smoothness assumptions.
\newblock {\em Ann. Statist.\/}~{\em 19\/}(2), 741--759.

\bibitem[\protect\citeauthoryear{Prince and Willsky}{Prince and
  Willsky}{1990}]{PrinceWillskyIEEE}
Prince, J.~L. and A.~S. Willsky (1990).
\newblock Reconstructing convex sets from support line measurements.
\newblock {\em IEEE Transactions on Pattern Analysis and Machine
  Intelligence\/}~{\em 12}, 377--389.

\bibitem[\protect\citeauthoryear{Schneider}{Schneider}{1993}]{Schneider}
Schneider, R. (1993).
\newblock {\em Convex Bodies: The Brunn-Minkowski Theory}.
\newblock Cambridge: Cambridge Univ. Press.

\bibitem[\protect\citeauthoryear{Stark and Yang}{Stark and
  Yang}{1998}]{stark1998vector}
Stark, H. and Y.~Yang (1998).
\newblock Vector space projections.
\newblock {\em John Wiley\&Sons, New York\/}.

\bibitem[\protect\citeauthoryear{Vitale}{Vitale}{1979}]{VitaleCC}
Vitale, R.~A. (1979).
\newblock Support functions of plane convex sets.
\newblock Technical report, Claremont Graduate School, Claremont, CA.

\bibitem[\protect\citeauthoryear{Wright}{Wright}{1981}]{Wright81}
Wright, F.~T. (1981).
\newblock The asymptotic behavior of monotone regression estimates.
\newblock {\em Ann. Statist.\/}~{\em 9\/}(2), 443--448.

\bibitem[\protect\citeauthoryear{Zhang}{Zhang}{2002}]{Zhang02}
Zhang, C.-H. (2002).
\newblock Risk bounds in isotonic regression.
\newblock {\em Ann. Statist.\/}~{\em 30\/}(2), 528--555.

\end{thebibliography}

\newpage

\appendix
\section{Some additional technical results and proofs}\label{apap}

In this appendix, we provide additional technical results and proofs.

\begin{proof}[Proof of Lemma \ref{genvitthm}]
    The inequality $h_{K^*}(\theta) \leq u(\theta, \phi)$ is obtained by
  using~\eqref{genvit} with $\alpha_1 = \theta+\phi, \alpha_2 = \theta-\phi$
  and $\alpha = \theta$. For $l(\theta,\phi) \leq h_{K^*}(\theta)$, we
  use~\eqref{genvit} with $\alpha_1 =  \theta+ 2\phi, \alpha_2 =
  \theta$ and $\alpha = \theta+\phi$ to obtain 
  \begin{equation*}
    h_{K^*}(\theta) \geq 2 h_{K^*}(\theta+\phi) \cos \phi - h_{K^*}(\theta+2\phi). 
  \end{equation*}
One similarly has $h_{K^*}(\theta) \geq 2 h_{K^*}(\theta-\phi) \cos
\phi - h_{K^*}(\theta-2\phi)$ and $l(\theta,\phi) \leq
h_{K^*}(\theta)$ is deduced by averaging these two inequalities.   
\end{proof}

\begin{lemma}\label{keydrop}
  Recall the quantity $\Delta_k(\theta_i)$ defined in
  \eqref{alex}. The inequality $\Delta_{2k}(\theta_i) \geq 1.5
  \Delta_k(\theta_i)$ holds for every $1 \leq i \leq n$ and $0 \leq k
  \leq n/16$.   
\end{lemma}
\begin{proof}
  We may assume without loss of generality that $\theta_i = 0$. We
  will simply write $\Delta_k$ for $\Delta_k(\theta_i)$ below for
  notational convenience. Let us define, for $\theta \in \R$, 
  \begin{equation*}
    \delta(\theta) := \frac{h_{K^*}(2 \theta) + h_{K^*}(-2\theta)}{2}
    - \frac{\cos 2\theta}{\cos \theta} \frac{h_{K^*}(\theta) +
      h_{K^*}(-\theta)}{2}. 
  \end{equation*}
  Note then that $\Delta_k = \sum_{j=0}^k \delta(2j\pi/n)/(k+1)$. We
  shall first prove that  
  \begin{equation}\label{keydrop.eq}
    \delta(y) \geq \left(\frac{\tan y}{\tan x}\right) \delta(x)
    \qt{for every $0 < y \leq \pi/4$ and $x < y \leq 2x$}. 
  \end{equation}
  For this, first apply \eqref{genvit} to $\alpha_1 = 2x, \alpha_2 =
  x$ and $\alpha = y$ to get
  \begin{equation*}
    h_{K^*}(y) \leq \frac{\sin(y-x)}{\sin x} h_{K^*}(2x) + \frac{\sin(2x-y)}{\sin
      x} h_{K^*}(x) . 
  \end{equation*}
  We then apply~\eqref{genvit} to $\alpha_1 = 2y, \alpha_2 = x$ and
$\alpha = 2x$ to get (note that $2y - x \leq 2y < \pi/2$)
\begin{equation*}
  h_{K^*}(2y) \geq \frac{\sin(2y - x)}{\sin x} h_{K^*}(2x) - \frac{\sin(2y -
    2x)}{\sin x} h_{K^*}(x). 
\end{equation*}
Combining these two inequalities, we get (note that $2y \leq
\pi/2$ which implies that $\cos 2y \geq 0$)
\begin{equation*}
  h_{K^*}(2y) - \frac{\cos 2y}{\cos y} h_{K^*}(y) \geq \alpha h_{K^*}(2x) - \beta h_{K^*}(x),
\end{equation*}
where
\begin{equation*}
  \alpha := \frac{\sin
      (2y-x)}{\sin x} - \frac{\cos 2y}{\cos y} \frac{\sin(y-x)}{\sin
      x}
\end{equation*}
and
\begin{equation*}
\beta := \frac{\sin(2y-2x)}{\sin x} + \frac{\cos
    2y}{\cos y} \frac{\sin(2x-y)}{\sin x}.
\end{equation*}
It can be checked by a straightforward calculation that
\begin{equation*}
  \alpha = \frac{\tan y}{\tan x} ~~~ \text{ and } ~~~ \beta =
  \frac{\tan y}{\tan x} \frac{\cos 2x}{\cos x}. 
\end{equation*}
It follows therefore that
\begin{equation*}
  h_{K^*}(2y) - \frac{\cos 2y}{\cos y}h_{K^*}(y) \geq \frac{\tan y}{\tan x}
  \left(h_{K^*}(2x) - \frac{\cos 2x}{\cos x} h_{K^*}(x)  \right). 
\end{equation*}
We similarly obtain
\begin{equation*}
  h_{K^*}(-2y) - \frac{\cos 2y}{\cos y}h_{K^*}(-y) \geq \frac{\tan y}{\tan x}
  \left(h_{K^*}(-2x) - \frac{\cos 2x}{\cos x} h_{K^*}(-x)  \right). 
\end{equation*}
The required inequality~\eqref{keydrop.eq} now results by adding the
above two inequalities. A trivial consequence of~\eqref{keydrop.eq} is
that $\delta(y) \geq \delta(x)$ for $0 < y \leq \pi/4$ and $x < y \leq
2x$. Further, applying~\eqref{keydrop.eq} to $y = 2x$ (assuming that $0 <
x < \pi/8$), we obtain $\delta(2x) \geq 2 \delta(x)$. Note that $\tan
2x = 2 \tan x/(1 - \tan^2x) \geq 2 \tan x$ for $0 < x < \pi/8$. 

To prove $\Delta_{2k} \geq (1.5) \Delta_k$, we fix $1 \leq k \leq
n/16$ (note that the inequality is trivial when $k = 0$) and note that  
\begin{equation*}
  \Delta_{2k} = \frac{1}{2k+1} \sum_{j = 0}^{2k} \delta
  \left(\frac{2j \pi}{n} \right) = \frac{1}{2k+1} \sum_{j = 1}^k
  \left(\delta \left(\frac{2(2j-1)\pi}{n} \right) + \delta
    \left(\frac{4j\pi}{n} \right) \right)
\end{equation*}
where we used the fact that $\delta(0)=0.$
Using the bounds proved for $\delta(\theta)$, we have
\begin{equation*}
  \delta \left(\frac{2(2j-1)\pi}{n} \right) \geq \delta \left(\frac{2j
      \pi}{n} \right) ~~ \text{ and } ~~   \delta
  \left(\frac{4j\pi}{n} \right) \geq 2 \delta \left(\frac{2j \pi}{n}
  \right). 
\end{equation*}
Therefore
\begin{equation*}
  \Delta_{2k} \geq \frac{3}{2k+1} \sum_{j=1}^k  \delta
  \left(\frac{2j\pi}{n} \right) \geq \frac{3}{2(k+1)} \sum_{j=0}^k  \delta
  \left(\frac{2j\pi}{n} \right) = \frac{3}{2}\Delta_k
\end{equation*}
and this completes the proof.
\end{proof}

\begin{lemma}\label{delkstar}
 Fix $i \in \{1, \dots, n\}$. Consider $\Delta_k(\theta_i)$ (defined
 in \eqref{alex}) and $k_*(i)$ (defined in \eqref{kst}). We then have
 the following inequalities
    \begin{equation}\label{delkstar.eq}
      \Delta_{k_*(i)}(\theta_i) \leq \frac{6(\sqrt{2} -
        1)\sigma}{\sqrt{k_*(i)+1}}.   
    \end{equation}
 and 
 \begin{equation}\label{an.delkstar.eq}
\Delta_k(\theta_i) \geq \max \left( \frac{(\sqrt{6} -
    2)\sigma}{\sqrt{k+1}}, \frac{(\sqrt{6} - 2) \sqrt{k+1}
    \sigma}{2(k_* + 1)}\right) 
\qt{for all $k > k_*(i), k \in \I$}.  
 \end{equation}
\end{lemma}

\begin{proof}
  Fix $i \in \{1,\dots, n\}$. Below we simply denote $k_*(i)$ and
  $\Delta_k(\theta_i)$ by $k_*$ and $\Delta_k$ respectively for
  notational convenience. 

  We first prove \eqref{delkstar.eq}. If $k_*\geq 2$, we have 
  \begin{equation*}
    \Delta_{k_*} + \frac{2\sigma}{\sqrt{k_*+1}} \leq \Delta_{k_*/2}
    + \sqrt{2} \frac{2\sigma}{\sqrt{k_*+2}} \leq \Delta_{k_*/2}
    + \sqrt{2} \frac{2\sigma}{\sqrt{k_*+1}}. 
  \end{equation*}
Using Lemma~\ref{keydrop} (note that $k_* \in \I$ and hence $k_* \leq
n/16$), we have $\Delta_{k_*/2} \leq (2/3)\Delta_{k_*}$. We therefore
have
  \begin{equation*}
    \Delta_{k_*} + \frac{2\sigma}{\sqrt{k_*+1}} \leq \frac{2}{3}\Delta_{k_*}
    + \sqrt{2} \frac{2\sigma}{\sqrt{k_*+1}}
  \end{equation*}
which proves~\eqref{delkstar.eq}. Inequality \eqref{delkstar.eq} is
trivial when $k_* = 0$. Finally, for $k_* = 1$, we have $\Delta_1 +
\sqrt{2} \sigma \leq \Delta_0 + 2 \sigma = 2 \sigma$ which again
implies \eqref{delkstar.eq}.

  We now turn to \eqref{an.delkstar.eq}. Let $k'$ denote the smallest
  $k \in \I$ for which $k > k_*$. We start by proving the first part
  of \eqref{an.delkstar.eq}:
  \begin{equation}\label{ffp}
    \Delta_k \geq \frac{(\sqrt{6} - 2)\sigma}{\sqrt{k+1}} \qt{for $k >
      k_*, k \in \I$}. 
  \end{equation}
  Note first that if \eqref{ffp} holds for $k = k'$, then it holds for
  all $k \geq k'$ as well because $\Delta_k \geq \Delta_{k'}$ (from
  Lemma \ref{keydrop}) and $1/\sqrt{k+1} \leq 1/\sqrt{k'+1}$. We
  therefore only need to verify \eqref{ffp} for $k = k'$. If $k_* =
  0$, then $k' = 1$ and because
  \begin{equation*}
    \Delta_1 + \frac{2\sigma}{\sqrt{2}} \geq \Delta_0 + 2 \sigma = 2 \sigma,
  \end{equation*}
  we obtain $\Delta_1 \geq (2 - \sqrt{2}) \sigma$. This implies
  \eqref{ffp}. On the other hand, if $k_* > 0$, then $k' =
  2k_*$ and we can write
  \begin{equation*}
    \Delta_{2k_*} + \frac{2 \sigma}{\sqrt{2k_* + 1}} \geq \Delta_{k_*}
    + \frac{2 \sigma}{\sqrt{k_* + 1}} \geq \frac{2 \sigma}{\sqrt{k_* +
        1}}. 
  \end{equation*}
  This gives 
  \begin{equation*}
    \Delta_{2k_*} \geq \frac{2 \sigma}{\sqrt{2k_* + 1}}
    \left(\sqrt{\frac{2k_* + 1}{k_* + 1}} - 1 \right)
  \end{equation*}
  which implies inequality \eqref{ffp} for $k = 2k_*$
  because $(2k_* + 1)/(k_* + 1) \geq 3/2$. The proof of \eqref{ffp} is
  complete.  

  For the second part of \eqref{an.delkstar.eq}, we use Lemma
  \ref{keydrop} which states $\Delta_{2k} \geq (1.5) \Delta_k \geq
  \sqrt{2} \Delta_k$ for all $k \in \I$. By a repeated application of
  this inequality, we get 
  \begin{equation*}
    \Delta_k \geq \sqrt{\frac{k}{k'}} \Delta_{k'} \geq
    \sqrt{\frac{k+1}{k'+1}} \Delta_{k'} \qt{for all $k \geq k'$}. 
  \end{equation*}
  Using \eqref{ffp} for $k = k'$, we get 
  \begin{equation*}
    \Delta_k \geq \frac{(\sqrt{6} - 2) \sigma \sqrt{k+1}}{k' + 1}. 
  \end{equation*}
  The proof of \eqref{an.delkstar.eq} is now completed by observing
  that $k' \leq 2k_* + 1$.  
\end{proof}

\begin{lemma}\label{varcal}
Fix $i \in \{1, \dots, n\}$. For every $0 \leq k \leq n/8$, the
variance of the random variable $\hat{U}_k(\theta_i)$ (defined in
\eqref{avig}) is at most $\sigma^2/(k+1)$. Also, for every $0 \leq k
\leq n/16$, the variance of the random variable
$\hat{\Delta}_k(\theta_i)$ (defined in \eqref{deluserep}) is at most
$\sigma^2/(k+1)$.   
\end{lemma}

\begin{proof}
Fix $1 \leq i \leq n$. We shall first prove the bound for the variance
of $\hat{U}_k(\theta_i)$ for a fixed $0 \leq k \leq n/8$. Note that 
\begin{equation*}
  \hat{U}_k(\theta_i) = \frac{1}{k+1} \sum_{j=0}^k \frac{Y_{i+j} +
    Y_{i-j}}{2 \cos (2 j \pi/n)}. 
\end{equation*}
It is therefore straightforward to see that 
\begin{equation*}
  \text{var}(\hat{U}_k(\theta_i)) = \frac{\sigma^2}{(k+1)^2} \left( 1
    + \frac{1}{2} \sum_{j=1}^k \sec^2(2 j \pi/n) \right). 
\end{equation*}
For $1 \leq j \leq k \leq n/8$, we have $\sec (2 j \pi/n) \leq
\sqrt{2}$ because $2 j \pi/n \leq \pi/4$. The inequality
$\text{var}(\hat{U}_k(\theta_i)) \leq \sigma^2/(k+1)$ then immediately
follows. 

Let us now turn to the variance of $\hat{\Delta}_k(\theta_i)$. When
$k=0$, the conclusion is obvious since $\hat{\Delta}_k(\theta_i) 
= 0$. Otherwise, the expression~\eqref{deluserep} for  
$\hat{\Delta}_k(\theta_i)$ can be rewritten as
\begin{equation*}
  \hat{\Delta}_k(\theta_i) = S_1 + S_2 + S_3 
\end{equation*}
where
\begin{equation*}
S_1 = \frac{-1}{k+1} \sum_{j=1}^k   \left\{ j \text{ is odd} \right\} 
  \frac{\cos (4j\pi/n)}{\cos (2j\pi/n)} \frac{Y_{i+j} + Y_{i-j}}{2},
\end{equation*}
\begin{equation*}
S_2 =  \frac{1}{k+1} \sum_{j=1}^k \left\{j \text{ is even} \right\}
\left(1 - \frac{\cos (4j\pi/n)}{\cos (2j\pi/n)} \right)
\frac{Y_{i+j} + Y_{i-j}}{2},
\end{equation*}
and
\begin{equation*}
  S_3 = \frac{1}{k+1} \sum_{j=k+1}^{2k}   \left\{j \text{ is even}
  \right\}  \frac{Y_j + Y_{-j}}{2}.
\end{equation*}
$S_1, S_2$ and $S_3$ are clearly independent. Moreover, the different
terms in each $S_i$ are also independent. Thus
\begin{equation*}
\text{var}(S_1) = \frac{\sigma^2}{2(k+1)^2} \sum_{j=1}^k   \left\{ j \text{
    is odd} \right\}  \frac{\cos^2 (4j\pi/n)}{\cos^2 (2j\pi/n)},
\end{equation*}
\begin{equation*}
  \text{var}(S_2) = \frac{\sigma^2}{2(k+1)^2} \sum_{j=1}^k   \left\{ j
    \text{ is even} \right\} \left(1 - \frac{\cos (4j\pi/n)}{\cos
      (2j\pi/n)} \right)^2,
\end{equation*}
and
\begin{equation*}
  \text{var}(S_3) = \frac{\sigma^2}{2(k+1)^2} \sum_{j=k+1}^{2k} \left\{ j 
    \text{ is even} \right\} \leq \frac{\sigma^2}{2(k+1)}. 
\end{equation*}
Now for $k \leq n/16$ and $1 \leq j \leq k$, 
\begin{equation*}
0 \leq  \frac{\cos (4j\pi/n)}{\cos (2j\pi/n)} \leq 1
\end{equation*}
which implies that $\text{var}(S_1) + \text{var}(S_2) \leq
\sigma^2/2(k+1)$. Thus $\text{var}(\hat{\Delta}_k(\theta_i)) \leq
\sigma^2/(k+1)$.  
\end{proof}

The following lemma was used in the proof of Theorem~\ref{lobo}. 
\begin{lemma}\label{useaux}
Let $\Delta_k$ be the quantity \eqref{alex} with $\theta_i = 0$ i.e., 
\begin{equation*}
\Delta_k :=  \frac{1}{k+1} \sum_{j=0}^k
  \left(\frac{h_{K^*}(4j\pi/n) + h_{K^*}(-4j
      \pi/n)}{2} - \frac{\cos(4j\pi/n)}{\cos(2j\pi/n)}
    \frac{h_{K^*}(2j\pi/n) + h_{K^*}(-2j\pi/n)}{2}
  \right).  
\end{equation*}
Then the following inequality holds for every $k \leq n/16$: 
\begin{equation*}
  \Delta_k \leq \frac{h_{K^*}(4k\pi/n) + h_{K^*}(-4k\pi/n)}{2\cos (4k\pi/n)} -
  h_{K^*}(0).
\end{equation*}
\end{lemma}
\begin{proof}
  From Lemma~\ref{keydrop}, it follows that $\delta(2i\pi/n) \leq
  \delta(2k\pi/n)$ for all $1 \leq i \leq k$ (this follows by
  reapplying Lemma~\ref{keydrop} to $2i\pi/n, 4i\pi/n, \dots$ until we
  hit $2k\pi/n$).  As a consequence, we have
  $\Delta_k \leq \delta(2k\pi/n)$. Now, if $\theta = 2k\pi/n$ then
  $\theta \leq \pi/8$ and we can write
  \begin{align*}
    \delta(\theta) &= \frac{h_{K^*}(2\theta) + h_{K^*}(-2\theta)}{2} - \frac{\cos
      2\theta}{\cos \theta} \frac{h_{K^*}(\theta) + h_{K^*}(-\theta)}{2} \\
&= \cos 2\theta \left(\frac{h_{K^*}(2\theta) + h_{K^*}(-2\theta)}{2\cos 2\theta} -
h_{K^*}(0)\right) - \cos 2\theta \left(\frac{h_{K^*}(\theta) + h_{K^*}(-\theta)}{2\cos
  \theta} - h_{K^*}(0) \right). 
  \end{align*}
Because $h_{K^*}(\theta) + h_{K^*}(-\theta) \geq 2 h_{K^*}(0) \cos \theta$ and $\cos
2\theta \geq 0$, we have 
\begin{equation*}
  \delta(\theta) \leq \cos 2\theta \left(\frac{h_{K^*}(2\theta) +
      h_{K^*}(-2\theta)}{2\cos 2\theta} - h_{K^*}(0) \right) \leq \frac{h_{K^*}(2\theta) +
      h_{K^*}(-2\theta)}{2\cos 2\theta} - h_{K^*}(0). 
\end{equation*}
The proof is complete. 
\end{proof}

\end{document}